\documentclass[12pt]{article}
\usepackage{amssymb,amsmath,amsthm,bm,setspace,mathrsfs}
\usepackage{url,color,units,dsfont,paralist}
\usepackage[T1]{fontenc}
\usepackage[margin=1in]{geometry}
\usepackage[font={footnotesize}]{caption}
\usepackage{graphicx,float}
\usepackage{subfigure,graphics}
\usepackage{xcolor,tikz}
\usepackage{tikz}\usetikzlibrary{shapes,arrows,patterns}
\include{rgb.tex}

\DeclareMathOperator{\Var}{Var}

\DeclareMathOperator{\Corr}{Corr}

\newcommand{\Z}{\mathbb{Z}}
\newcommand{\Q}{\mathbb{Q}}
\newcommand{\R}{\mathbb{R}}

\newcommand{\cS}{\mathcal{S}}
\newcommand{\cV}{\mathcal{V}}
\newcommand{\Dimh}{{\rm Dim}_{_{\rm H}}}

\newcommand{\be}{\begin{equation}}
\newcommand{\ee}{\end{equation}}

\newcommand{\lip}{\text{\rm Lip}_\sigma }

\renewcommand{\P}{\mathrm{P}}
\newcommand{\E}{\mathrm{E}}

\newcommand{\1}{\boldsymbol{1}}
\renewcommand{\d}{{\rm d}}

\newcommand{\e}{{\rm e}}
\renewcommand{\geq}{\geqslant}
\renewcommand{\leq}{\leqslant}
\renewcommand{\ge}{\geqslant}
\renewcommand{\le}{\leqslant}

\author{Davar Khoshnevisan\\University of Utah
\and Kunwoo Kim\\POSTECH
\and Yimin Xiao\\Michigan State University}
\title{\bf A macroscopic multifractal analysis\\of parabolic stochastic PDEs\thanks{
	Research supported in part by the NSF grants DMS-1307470, DMS-1608575
	and DMS-1607089 [D.K. \&\ Y.X.]
	and 0932078000 [K.K. through The
	Mathematical Sciences Research Institute at UC Berkeley], and  the National 
	Research Foundation of Korea (NRF-2017R1C1B1005436) and the TJ Park Science Fellowship of POSCO
TJ Park Foundation [K.K]. A portion of this material is based upon work supported also
	by the NSF Grant  DMS-1440140 while D.K. was
	in residence at the Mathematical Sciences Research Institute in Berkeley, CA.}}
\date{May 13, 2017}

\newtheorem{stat}{Statement}[section]
\newtheorem{proposition}[stat]{Proposition}
\newtheorem{corollary}[stat]{Corollary}
\newtheorem{theorem}[stat]{Theorem}
\newtheorem{lemma}[stat]{Lemma}
\theoremstyle{definition}

\newtheorem{remark}[stat]{Remark}

\numberwithin{equation}{section}
\begin{document}%\onehalfspacing
\maketitle
\begin{abstract}
	It is generally argued that the solution to a  stochastic PDE with multiplicative noise---
	such as $\dot{u}=\frac12 u''+u\xi$, where $\xi$ denotes
	space-time white noise---routinely produces exceptionally-large peaks
	that are ``macroscopically multifractal.'' See, for example,
	Gibbon and Doering (2005), Gibbon and Titi (2005), and
	Zimmermann et al (2000). A few years ago, we proved that the
	spatial peaks of the solution to the mentioned stochastic PDE
	indeed form a random multifractal in the macroscopic sense of
	Barlow and Taylor (1989; 1992). The main result of the present paper is
	a proof of a rigorous formulation of the assertion that the spatio-temporal peaks of
	the solution form infinitely-many different multifractals
	on infinitely-many different scales, which we sometimes refer to as
	``stretch factors.'' A simpler, though still complex,
	such structure is shown to also exist for the constant-coefficient version of
	the said stochastic PDE.\\


\end{abstract}%\newpage

\section{Introduction}

\subsection{The main result}
Let $\xi$ denote space-time white noise,  normalized so that
\[
	\text{\rm Cov}[\xi(t\,,x)\,,\xi(s\,,y)] = \delta_0(t-s)\cdot\delta_0(x-y)
	\hskip0.6in
	\text{for all $s,t\ge 0$ and $x,y\in\R$},
\]
and consider, throughout the paper, the stochastic heat equation
\begin{equation}\label{SHE}
	\dot{u}(t\,,x) = \tfrac12 u''(t\,,x) + \sigma(u(t\,,x))\xi(t\,,x),
\end{equation}
defined on $(t\,,x)\in(0,\infty)\times\R$ with initial datum $u(0)=u_0\in L^\infty(\R)$.
We always will assume that $u_0$ and $\sigma$ are nonrandom real-valued
functions on $\R$, that $\sigma$ is Lipschitz continuous and satisfies
$\sigma(0)=0$, and that $\inf_{x\in\R}u_0(x)>0$.
Among many other things, these conditions ensure that \eqref{SHE} has a unique
continuous and strictly-positive solution $u(t\,,x)$ that has finite
moments of all order, uniformly in
$x$ and locally uniformly in $t$ \cite{Dalang,DZ,Minicourse,CBMS,Mueller,Walsh}.

The main objective of this paper is to study the intermittency properties of
the solution to the stochastic PDE \eqref{SHE}.
In order to recall the meaning of this phrase, let us first define
\[
	\underline\gamma(k) := \liminf_{t\to\infty} t^{-1}
	\log\inf_{x\in\R}\E\left(|u(t\,,x)|^k\right)
	\quad\text{and}\quad
	\overline\gamma(k) := \limsup_{t\to\infty} t^{-1}
	\log\sup_{x\in\R}\E\left(|u(t\,,x)|^k\right),
\]
for all $k\in[2\,,\infty)$. The functions $\underline\gamma$
and $\overline\gamma$ are known respectively as the
\emph{lower} and the \emph{upper moment Lyapunov exponents}
of the solution $u$ to \eqref{SHE}. According to Jensen's inequality,
both $\underline\gamma$ and $\overline\gamma$ are nondecreasing
functions on $[2\,,\infty)$. The solution $u$ to \eqref{SHE}
is said to be \emph{intermittent} if $\underline\gamma$ and
$\overline\gamma$ are both \emph{strictly} increasing on
$[2\,,\infty)$; see  \cite{FK}
and  \cite{BC,CM94,GD,Mandelbrot,Molch91,ZRS} for earlier variations.

It is known that the solution to \eqref{SHE} is intermittent \cite{FK}
under the additional constraint that $\sigma$ satisfies the following condition:
\begin{equation}\label{cond:interm}
	\inf_{|z|>0}| \sigma(z)/z|>0.
\end{equation}
For this reason, we might refer to Condition \eqref{cond:interm}
as an ``intermittency condition.''
The intermittency
condition \eqref{cond:interm} also has quantitative consequences.
For example, it implies---see
\cite{BC,Chen1,Chen2,FK,JKM}---that there exist
finite and positive constants $M_0<N_0$ and $M<N$ such that
\begin{equation}\label{eq:moments:SHE}
	M_0^k\e^{Mk^3 t}\le
	\E\left(|u(t\,,x)|^k\right) \le N_0^k\e^{Nk^3t},
\end{equation}
uniformly for all real numbers $x\in\R$, $t>0$, and $k\ge2$.
It follows from these bounds that
\[
	M k^3\le \underline\gamma(k)\le\overline\gamma(k)\le Nk^3
	\qquad\text{for all $k\in[2\,,\infty)$.}
\]

Also, \eqref{eq:moments:SHE} suggests that the tall spatio-temporal
peaks of the stochastic process $u$ might grow exponentially with time.
For a heuristic argument see the Introductions of
Bertini and Cancrini \cite{BC} and Camona and Molchanov
\cite{CM94}, together with Chapter 7 of Khoshnevisan \cite{CBMS}.
With this connection to spatio-temporal peaks
in mind, let us consider the random space-time set,
\[
	\mathscr{P}(\beta) :=\left\{ (x\,,t)\in\R\times(\e\,,\infty):
	u(t\,,x) > \e^{\beta t}\right\},
\]
of peaks of height profile $t\mapsto\exp(\beta t)$
for every $\beta>0$. When $t\gg1$ and $(x\,,t)\in\mathscr{P}(\beta)$,
we have $u(t\,,x)>\e^{\beta t}\gg1$. The present work, in a sense,
implies that a ``typical'' such pair $(x\,,t)$
in fact satisfies $u(t\,,x)\approx \e^{\beta t}$. Thus, we see that if
$\mathscr{P}(\beta)\neq\varnothing$ a.s.\
for infinitely-many distinct $\beta>0$,
then there are infinitely-many natural length scales in which one can measure
the tall peaks of the solution to \eqref{SHE}. This will verify, quantitatively, a property
that is believed to hold for a large class of ``complex systems''; see Gibbon and Titi
\cite{GibbonTiti} for an argument.

In fact, the situation is more complicated still.
For every $\vartheta>0$ let us define a function
$S_\vartheta:\R\times(0\,,\infty)\to\R\times(1\,,\infty)$ as follows:
\[
	S_\vartheta(x\,,t) := \left( x\,,\e^{t/\vartheta}\right)\qquad
	\text{for all $(x\,,t)\in\R\times(0\,,\infty)$}.
\]
It is easy to see that, for small values of $\vartheta$,
the application $S_\vartheta$
amounts to a nonlinear stretching of
 $\R\times(0\,,\infty)$ in the $t$-direction. For instance,
 \[
 	S_\vartheta\left( [1\,,2]\times[1\,,2]\right) = [1\,,2]\times\left[
	\e^{1/\vartheta},\e^{2/\vartheta}\right].
 \]
 In this way we can see that, when $\vartheta$ is small, $S_\vartheta$
 maps the upright square box $[1\,,2]\times[1\,,2]$ to an elongated, stretched,
 box in the $(x\,,t)$ plane.

Let $\Dimh$ denote the Barlow--Taylor
\cite{BT1989,BT1992} macroscopic Hausdorff dimension on $\R^2$.
[For a detailed definition of $\Dimh$ see \S\ref{sec:Dim} below.]
The following is the main result of this paper.

\begin{theorem}\label{th:SHE}
	If $\sigma$ satisfies the intermittency condition \eqref{cond:interm},
	then there exist finite constants $A>a>0$ and $b,\varepsilon>0$
	such that
	\[
		 2-  A\beta^{3/2}\vartheta\le
		 \Dimh\left[S_\vartheta\left(\mathscr{P}(\beta)\right)\right]
		\le 2-a\beta^{3/2}\vartheta
		\qquad\text{a.s.,}
	\]
	valid for every $\beta>b$ and $\vartheta\in(0\,,\varepsilon\beta^{-3/2})$.
\end{theorem}

The following corollary of Theorem \ref{th:SHE}
clarifies the intent of that theorem.

\begin{corollary}\label{co:SHE}
	If \eqref{cond:interm} holds then there exist nonrandom
	numbers $0<\beta_1<\beta_2<\cdots$ and $0<\vartheta_1<\vartheta_2<\cdots$
	such that:
	\begin{itemize}
		\item $1< \Dimh [S_{\vartheta_i} (\mathscr{P}(\beta_j) ) ]<2$ a.s.\ for
			all $i,j\ge 1$; and
		\item If there exist $i,j,k,l\ge 1$ such that
			$\Dimh [S_{\vartheta_i} (\mathscr{P}(\beta_j) ) ]
			= \Dimh [S_{\vartheta_k} (\mathscr{P}(\beta_l) ) ]$ with positive probability,
			then $i=k$ and $j=l$.
	\end{itemize}
\end{corollary}

Corollary \ref{co:SHE} shows, in particular, that
there exist infinitely-many different stretch factors $\vartheta_1,\vartheta_2,\ldots$
and infinitely-many different length scales $\beta_1,\beta_2,\ldots$ such that
for every $i\ge 1$, the $\vartheta_i$-stretching of peaks of height
$t\mapsto\exp\{\beta_j t\}$ [$j=1,2,\ldots$]
all have distinct and nontrivial macroscopic Hausdorff dimension. This means that
every $\vartheta_i$-stretching of the tall peaks of $u$ is macroscopically
multifractal. Moreover, the
said Hausdorff dimensions themselves are distinct as we vary the stretch factors.
One can interpret this finding as follows: Under the intermittency condition \eqref{cond:interm},
the tall peaks  of the solution to \eqref{SHE} form different multifractals on infinitely-many
different stretch scales. We believe that the quantitative statement of Theorem \ref{th:SHE}
and its proof are novel. However, the  idea that the peaks of
$u$ should form very complex macroscopic space-time multifractals has been argued
much earlier in the literature. We learned that idea from
an insightful paper by Gibbon and Doering \cite{GD} on the role of intermittency in turbulence.
And a paper by Zimmerman et al \cite{Zimm} discusses this sort of complex
multifractal behavior in the context of the
closely-related stochastic Allen--Cahn equation with multiplicative forcing.

In a recent paper \cite{KKX} we have established that,
at each fixed time $t>0$, the solution to \eqref{SHE} under \eqref{cond:interm}
and $\sigma\equiv 1$ are both multifractal. This is a somewhat
counterintuitive statement because
the solution to \eqref{SHE} is intermittent under \eqref{cond:interm}---this is close to the
KPZ universality class \cite{ACQ,Corwin,Kardar,KPZ}---and is
non intermittent when $\sigma\equiv1$---this
is close to the Edwards-Wilkinson universality class \cite{ACQ,Corwin,EW}.
We shall see in \S\ref{sec:non:interm} below [see Theorem \ref{th:lin}]
that,
as can be determined by Hausdorff dimension considerations alone,
the spatio-temporal peaks of the constant-$\sigma$ case are significantly smaller
than the spatio-temporal peaks of the case \eqref{cond:interm}, though both
models have infinitely-many different natural length scales and stretch factors.

\subsection{An outline of the proof of Theorem \ref{th:SHE}}

The proof of Theorem \ref{th:SHE} hinges on a blend of probabilistic,
analytic, and geometric ideas, many of which we believe are novel.
It also relies on various probability estimates of our earlier paper
\cite{KKX}, which we will recall in due time.

In the remainder of this introduction we
outline the intuition behind the proof of Theorem \ref{th:SHE}, though
the proper
proof itself contains a number of additional technical hurdles that
will need to be circumvented.

First, we present the following elegant geometric
result, whose proof appears in the next section of the paper.

\begin{proposition}\label{pr:Epi}
	Suppose $f:(0\,,\infty)\to(0\,,\infty)$ is a strictly increasing convex function,
	and recall that
	the \emph{epigraph} of $f$ in $(0\,,\infty)^2$ is the planar set,
	\[
		\text{\rm Epi}[f] := \left\{ (x\,,y)\in(0\,,\infty)^2:\ y\ge f(x)\right\}.
	\]
	If
	\begin{equation}\label{liminf:f}
		\liminf_{x\to\infty}\,\frac{f(x)}{x}>1,
	\end{equation}
	then
	\begin{equation}\label{dimtest}
		\Dimh\left( \text{\rm Epi}[f]\right) =
		1 +\inf \left\{0<\alpha\leq 1:\
		\int_1^\infty \left( \frac{f^{-1}(x)}{x}\right)^\alpha\frac{\d x}{x}<\infty\right\},
	\end{equation}
	where $\inf\varnothing:=1$.
\end{proposition}

Proposition \ref{pr:Epi}, and its proof, together will imply the following.

\begin{corollary}\label{co:Epi}
	Let $\log_+ y:=\log(y\vee \e)$ for all $y\ge0$,
	and for all $p>1$ and $q> 0$ define
	\[
		\mathcal{E}_p:= \left\{ (x\,,y)\in(0\,,\infty)^2:\, y\ge x(\log_+ x)^p \right\}
		\quad \text{and} \quad \widetilde{\mathcal{E}_q}:=
		\left\{ (x\,,y)\in(0\,,\infty)^2:\, y\ge x^q \right\}.
	\]
	Then,
	\[
		\Dimh(\mathcal{E}_p)=1+ p^{-1}
%\min\left(1\,,p^{-1}\right)
		\quad\text{and}\quad
		\Dimh(\widetilde{\mathcal{E}_q})=\begin{cases}
			2&\text{if $0< q\le1$},\\
			1&\text{if $q>1$}.
		\end{cases}
	\]
\end{corollary}
Choose and fix an arbitrary $q>1$, as close to one as one would like.
According to Corollary \ref{co:Epi},  $\Dimh(\widetilde{\mathcal{E}}_q)=1$;
this and elementary properties of macroscopic Hausdorff dimension readily
imply that
\[
	\Dimh \left( S_\vartheta(\mathscr{P}(\beta))\cap\widetilde{\mathcal{E}}_q\right)
	\le\Dimh ( \widetilde{\mathcal{E}}_q )=1
	\qquad\text{a.s.}
\]
One may notice next that Theorem \ref{th:SHE} asserts that
$\Dimh[S_\vartheta(\mathscr{P}(\beta))]\ge1$ a.s.\
over the range of $\vartheta$ and $\beta$ mentioned in the statement
of the theorem.  Thus,
it follows that all of the interesting fractal behavior of $S_\vartheta(\mathscr{P}(\beta))$
occurs off the infinite set $\widetilde{\mathcal{E}}_q$. In other words,
in order to understand the large-scale fractal structure of
$S_\vartheta(\mathscr{P}(\beta))$, it is necessary and sufficient to
understand the large-scale fractal structure of
the random set
\[
	S_\vartheta\left(\mathscr{P}(\beta))\cap
	\left(\R^2\setminus\widetilde{\mathcal{E}}_q\right)\right).
\]
A simple symmetry calculation reduces this problem to one about the analysis
of the random set
\[
	\left(\R_+\times\R_+\right)\cap S_\vartheta\left(\mathscr{P}(\beta))\cap
	\left(\R^2\setminus\widetilde{\mathcal{E}}_q\right)\right)
	=: \R^2_+\setminus\mathcal{R}_1.
\]
Figure \ref{fig:corner} includes a depiction of the restriction of
$\mathcal{R}_1$ to a large box $\cup_{i=1}^n\cS_i$,
where
\[
	\cS_n := \left[\e^n,\e^{n+1}\right]\times \left[\e^n,\e^{n+1}\right]
	\qquad[n\gg1].
\]
We might think of $\cS_n$ as the ``$n$th shell.''

The preceding discussion tells us
that---as far as the macroscopic structure of $S_\vartheta(\mathscr{P}(\beta))$
is concerned---nothing interesting happens in $\mathcal{R}_1$.

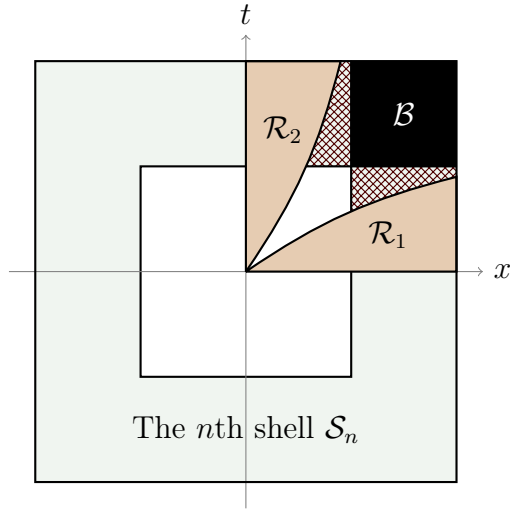
\begin{figure}[h!]\centering
	\begin{tikzpicture}[scale=.7,domain=-1.05:1.05]%\draw[help lines] (-4,-4) grid (4,4);
		\filldraw[green!20!gray!10] (-4,-4) rectangle (4,4);
		\filldraw[white] (-2,-2) rectangle (2,2);
		\filldraw[pattern color = black!70!red,pattern=crosshatch ](2,0) rectangle (4,2);
		\draw[pattern color = black!70!red,pattern=crosshatch ](0,2) rectangle (2,4);
		\draw[thick,black] (-2,-2) rectangle (2,2);
		\filldraw[black](2,2) rectangle (4,4);
		\draw[->,gray] (-4.5,0) -- (4.5,0) node[right] {\textcolor{black}{$x$}};
		\draw[->,gray] (0,-4.5) -- (0,4.5) node[above] {\textcolor{black}{$t$}};
		\filldraw[thick,fill=brown!40,draw=black] (0,0) to[bend right=10] (1.8,4) to (0,4) to (0,0);
		\filldraw[thick,fill=brown!40,draw=black] (0,0) to[bend left=10](4,1.8) to (4,0) to (0,0);
		\draw[thick,black] (-4,-4) rectangle (4,4);
		\draw[thick] (2,2) -- (2,4);
		\draw[thick] (2,2) -- (4,2);
		\node at (0,-3) {The $n$th shell $\cS_n$};
		\node at (.7,2.7) {$\mathcal{R}_2$};
		\node at (2.7,.7) {$\mathcal{R}_1$};
		\node[color=white] at (3,3) {$\mathcal{B}$};
	\end{tikzpicture}
	\caption{Searching for elements of $S_\vartheta(\mathscr{P}(\beta))$
		in the top-right quadrant in the $n$th shell $\cS_n$.}
	\label{fig:corner}
\end{figure}
Let $\mathcal{R}_2$ denote the symmetric reflection of  $\mathcal{R}_1$
about the diagonal of $\R^2_+$; see Figure \ref{fig:corner}. A second
symmetry calculation shows that nothing interesting happens in $\mathcal{R}_2$.
Thus, it follows that all of the interesting large-scale fractal structure of
$S_\vartheta(\mathscr{P}(\beta))$ is contained in the part of
$\R^2_+$ that is sandwiched between $\mathcal{R}_2$ and $\mathcal{R}_1$;
that is, to $\R^2_+\setminus(\mathcal{R}_2\cup\mathcal{R}_1)$.

Next, a covering argument can be devised to reduce the domain of interest
from the relatively complicated infinite set
$\R^2_+\setminus(\mathcal{R}_2\cup\mathcal{R}_1)$ to the much simpler infinite
set
\[
	\mathcal{B} := \bigcup_{n=1}^\infty \left(\left[\e^n,\e^{n+1}\right]\times
	\left[\e^n,\e^{n+1}\right]\right) \subsetneq
	\R^2_+\setminus(\mathcal{R}_2\cup\mathcal{R}_1).
\]
The restriction of the unbounded set $\mathcal{B}$ to the $n$th shell
appears in Figure \ref{fig:corner} as a black upright
square. In this way, we see that the proof of Theorem \ref{th:SHE}
is reduced to proving that
\begin{equation}\label{eq:reduced}
	 \max\left\{1\,,2-  A\beta^{3/2}\vartheta\right\}\le
	 \Dimh\left[S_\vartheta\left(\mathscr{P}(\beta)\right)\cap
	 \mathcal{B}\right]
	\le\max\left\{ 1\,,2-a\beta^{3/2}\vartheta\right\}
	\qquad\text{a.s.,}
\end{equation}
for all $\beta$ sufficiently large and $\vartheta$ sufficiently small.
This reduction is significant because every point
$(x\,,t)$ in the restriction of $\mathcal{B}$ to the $n$th shell
has the property that $\e^{-1} \le x/y \le \e$. That is,
the spatial behavior and the temporal behavior of $u$ in $
\mathcal{B}$ are, in some sense, comparable. This property turns out
to, in some sense, help ``homogenize'' our problem on large scales.

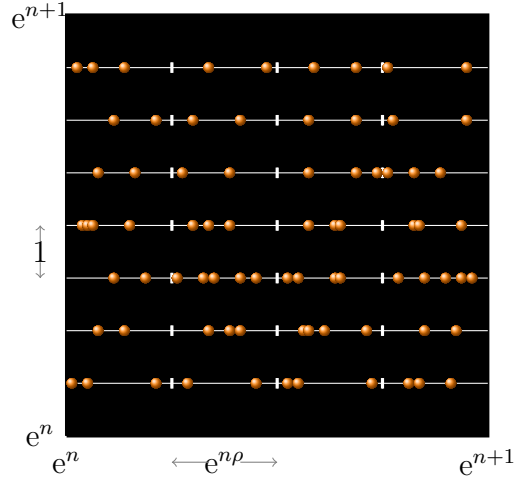
\begin{figure}[h!]\centering
	\begin{tikzpicture}[scale=.7,domain=-1.05:1.05]
		%\draw[step=.5,gray,very thin] (0,0) grid (8,8);
		\filldraw[very thick] (0,0) to (8,0) to (8,8) to (0,8) to (0,0);
		\draw[thin,color=white](0,1) to (8,1);
		\draw[thin,color=white](0,2) to (8,2);
		\draw[thin,color=white](0,3) to (8,3);
		\draw[thin,color=white](0,4) to (8,4);
		\draw[thin,color=white](0,5) to (8,5);
		\draw[thin,color=white](0,6) to (8,6);
		\draw[thin,color=white](0,7) to (8,7);
		
		\foreach \y in {1,2,3,4,5,6,7}
			\foreach \x in {2,4,6}
			\draw[thick,color=white](\x-0.01,\y-0.1) to (\x-0.01,\y+0.1);
		\foreach \y in {1,2,3,4,5,6,7}
			\foreach \x in {2,4,6}
			\draw[thick,color=white](\x+0.01,\y-0.1) to (\x+0.01,\y+0.1);
			
		\node at (0,-0.5) {$\e^n$};\node at (8,-0.5) {$\e^{n+1}$};
		\node at (-0.5,0) {$\e^n$};\node at (-0.5,8) {$\e^{n+1}$};
		\node at (-0.5,3.5) {$1$};
		\draw[->,gray] (-0.5,3.7) to (-0.5,4);
		\draw[->,gray] (-0.5,3.3) to (-0.5,3);
		\node at (3,-0.5) {$\e^{n\rho}$};
		\draw[->,gray] (3.3,-0.5) to (4,-0.5);
		\draw[->,gray] (2.7,-0.5) to (2,-0.5);

		\foreach \x in {0.1,0.4,1.7,2.3,3.6,4.2,4.4,5.8,6.5,6.7,7.3}
			\shade[ball color=orange] (\x,1) circle (.1);
		\foreach \x in {0.6,1.1,2.7,3.1,3.3,4.5,4.6,4.9,5.7,6.8,7.4}
			\shade[ball color=orange] (\x,2) circle (.1);
		\foreach \x in {0.9,1.5,2.1,2.6,2.8,3.3,3.6,4.2,4.4,5.1,5.2,6.3,6.8,7.2,7.5,7.7}
			\shade[ball color=orange] (\x,3) circle (.1);
		\foreach \x in {0.3,0.4,0.5,1.2,2.4,2.7,3.1,4.6,5.1,5.2,6.6,6.7,7.5}
			\shade[ball color=orange] (\x,4) circle (.1);
		\foreach \x in {0.6,1.3,2.2,3.1,4.6,5.5,5.9,6.1,6.6,7.1}
			\shade[ball color=orange] (\x,5) circle (.1);
		\foreach \x in {0.9,1.7,2.4,3.3,4.6,5.5,6.2,7.6}
			\shade[ball color=orange] (\x,6) circle (.1);
		\foreach \x in {0.2,0.5,1.1,2.7,3.8,4.7,5.5,6.1,7.6}
			\shade[ball color=orange] (\x,7) circle (.1);
	\end{tikzpicture}
	\caption{The distribution of tall peaks in $\mathcal{B}$. Circles
		represent tall peaks.}
	\label{fig:peaks}
\end{figure}
In order to prove \eqref{eq:reduced}, we now enlarge our
view of $\mathcal{B}$---see Figure \ref{fig:peaks}---and
do a multiscale analysis in $\mathcal{B}$.

For every large integer $n\gg1$, let us subdivide,
in the $t$ direction, the box $\mathcal{B}\cap[\e^n,\e^{n+1}]$
using equally-spaced lines $L_1,L_2,\ldots$ that are one unit apart
[in the $t$, or vertical,  direction]. It turns out
that there exist  two numbers
$\rho_L=\rho_L(\vartheta\,,\beta) \in(0\,,1)$ and 
$\rho_U=\rho_U(\vartheta\,,\beta) \in(0\,,1)$  with $\rho_L\leq \rho_U$ such that
 the following happens almost surely for all $n$ large: 
\begin{compactenum}
	\item For every $\rho\in(\rho_U\,,1)$, the equipartition
		of every $L_j$ into subintervals of length $\exp(n\varrho)$
		has the property that all of the said subintervals contain
		at least one point where the peak of $u$ is of height $\ge\exp\{\beta t\};$ whereas
	\item For every $\rho\in(0\,,\rho_L)$, none of the mentioned
		subintervals correspond to  a peak of height $\ge\exp\{\beta t\}$.
\end{compactenum}
In other words, $S_\vartheta(\mathscr{P}(\beta))\cap\mathcal{B}$
behaves, on large scales, as a ``random self-similar fractal.'' It is easy to
compute the macroscopic Hausdorff dimension of a self-similar
fractal; a variation on that calculation then yields \eqref{eq:reduced}.\\

\subsection{A brief outline of the paper}
Let us conclude the Introduction by describing briefly the structure of the paper.

In \S2 we recall some basic facts about the Barlow--Taylor theory of
macroscopic fractals, macroscopic Hausdorff dimension, etc. \cite{BT1989,BT1992}.
Proposition \ref{pr:Epi} and Corollary \ref{co:Epi} are also proved in \S2.

Section 3 is dedicated to the proof of Theorem \ref{th:SHE} and its
Corollary \ref{co:SHE}. The results of this section include large-deviations
probability bounds, localization estimates, and bounds on a so-called
spatial correlation length of the solution to \eqref{SHE}. It is generally
believed that the solution to \eqref{SHE} spatially decorrelates at length scale
$\Theta(t^{3/2})$  when $t\gg1$.\footnote{Recall that $f(t)=\Theta(g(t))$
means that there exists $a>1$ such that
$a^{-1} g(t)\le f(t) \le a g(t)$ for all sufficiently large $t$.}
We have not found a carefully-stated form
of this as a conjecture in print, but the fact is for example hinted at implicitly in
Corwin \cite{Corwin}, and is also believed to be true by many physicists.
Here, we prove that the said correlation length
is not more than $\Omega(t^2)$;\footnote{Recall that
$f(t)=\Omega(g(t))$ means that there exists a positive constant $a$ such that
$f(t)\ge ag(t)$ for all sufficiently large $t$.}
%\textcolor{red}{(the definition of $\Omega$ is on the footnote on p.19.
%Should we write it here?)};
for a careful statement see Theorem \ref{th:localization}. This result is
the best-known bound to date on the correlation length of $x\mapsto
u(t\,,x)$ when $t\gg1$.

For purposes of comparison, we derive in \S4 an analogue of Theorem \ref{th:SHE}
that holds for the solution to \eqref{SHE} in the case that $\sigma$ is constant.
The main theorem of that section is Theorem \ref{th:lin} which implies that,
when $\sigma$ is a constant,
the exceptionally-tall spatio-temporal peaks of the solution to \eqref{SHE}
are much smaller than when for example $\sigma(u)=u$. But the complex
multifractal structure of the peaks continues to pervade.

At the end of the paper we have taken care to collect a list of
many of the constants that appear within proofs, particularly those
of Theorem \ref{th:SHE}. It turns out that one has to be very careful
in some cases in order to make sure that various parameter dependencies
do not arise. In some cases, this is a truly non-trivial task; therefore,
we have taken care to outline  the important universal constants,
together with where they first arise, in Table \ref{table}
in an appendix that follows the bibliography. In this way one can use
Table \ref{table} in order to keep track of
the various parameter dependencies of interest.
	
\section{Macroscopic dimension}\label{sec:Dim}
Let us begin by recalling the Barlow--Taylor theory of macroscopic Hausdorff
dimension \cite{BT1989,BT1992}.

For every integer $k\ge0$, let
\begin{equation}\label{VS}
	\cV_k :=\left(-\e^k\,,\e^k\right]^2,
	\quad \cS_0:=\cV_0,\quad
	\cS_{k+1} := \cV_{k+1}\setminus\cV_k.
\end{equation}
Also, for every $n\in\Z$ let
$\mathscr{D}^{(n)}$ denote the collection of all \emph{$\e$-adic squares}
of the form
\begin{equation}\label{square}
	Q^{(n)}:=(i\e^n\,,(i+1)\e^n]\times (j\e^n\,,(j+1)\e^n],
\end{equation}
where $i,j\in\Z$ range over all integers.
If a square $Q^{(n)}$ has the form \eqref{square}, then
we say that $(i\e^n,j\e^n)$ is the \emph{southwest corner}
of $Q^{(n)}$, and $\e^n$ is the \emph{sidelength} of $Q^{(n)}$.
By $\mathscr{D}$ we mean the collection of
all $\e$-adic squares of $\R^2$; that is,
\[
	\mathscr{D}:=\bigcup_{n=-\infty}^\infty\mathscr{D}_n.
\]
A special role is played by
\[
	\mathscr{D}_{\ge 1}:=\bigcup_{n=0}^\infty
	\mathscr{D}_n.
\]
This is the collection of all $\e$-adic squares of sidelength
not smaller than $ 1$.

For every integer $k\ge 0$,
all $\rho\in (0\,,\infty)$, and each $A \subseteq \R^d$ define
\begin{equation}\label{nu}
	\nu^k_\rho(A) :=
	\min \sum_{i=1}^{m} \e^{\rho(\ell_i -k-1)},
\end{equation}
where the minimum is over all possible coverings of
$A \cap \cS_k$ by $\e$-adic  squares $Q_1,\ldots,Q_m\subseteq \cS_k$ of sidelength
$\e^{\ell_i} \ge 1$.
Note in particular that these squares are all elements of $\mathscr{D}_{\ge 1}$.

M. T. Barlow and S. J. Taylor \cite{BT1989,BT1992} defined the
\textit{macroscopic Hausdorff dimension} of a set $A\subseteq\R^2$ as
\[
	\Dimh(A) := \inf\left\{ \rho>0:\ \sum_{k=1}^\infty
	\nu^k_\rho(A)< \infty\right\}.
\]
The papers by Barlow and Taylor \cite{BT1989,BT1992} contain further
information about the macroscopic Hausdorff dimension $\Dimh$.
Among other things, the following result of Barlow and Taylor is noteworthy.

\begin{proposition}[Barlow and Taylor \cite{BT1992}]\label{pr:BT:Dim}
	Let $A\subseteq\R^2$ be a set.
	\begin{compactenum}
	\item Suppose we redefined $\nu^k_\rho(A)$ as in \eqref{nu},
		but where the minimum is over all possible coverings of
		$A \cap \cS_k$ by squares $Q_1,\ldots,Q_m\subseteq\cS_k$
		of the form $[x_1\,, x_1+r)\times [x_2\,, x_2+r)$,
		where $r\ge1$.
		Then this change does not alter the numerical value of $\Dimh(A)$.
	\item Choose and fix a real number $a>1$, and
		suppose we redefined $\{\cV_k\}_{k\ge0}$---and hence also $\{\cS_k\}_{k\ge0}$---%
		in \eqref{VS} as follows: $\cV_k :=\left(a^k\,,a^{k+1}\right]^2$
		for every integer $k\ge0$. Then, this change does not affect the numerical value
		of $\Dimh(A)$.
	\end{compactenum}
\end{proposition}

Having dispensed with an introduction
to macroscopic Hausdorff dimension, we establish Proposition \ref{pr:Epi} next.

\begin{proof}[Proof of Proposition \ref{pr:Epi}]
	We first prove the proposition under the more restrictive condition,
	\begin{equation}\label{liminf:f:e}
		\liminf_{x\to\infty}\,\frac{f(x)}{x}>\e.
	\end{equation}
	We then outline how to replace the preceding by the weaker condition \eqref{liminf:f}.
	
	As is commonly done in the local theory of dimension,
	one proceeds by first obtaining an upper bound and then a lower
	bound for $\Dimh(\text{\rm Epi}[f])$.
	We first consider the upper bound for
	$\Dimh( \text{\rm Epi}[f])$.
	
	Define
	\[
		f_n := f^{-1}(\e^{n+1})\qquad\text{for all integers $n\ge 0$}.
	\]
	Because  $f$ is strictly increasing, one can draw a picture---see Figure
	\ref{fig:range(f)}---in order to see that
	\[
		\text{\rm Epi}[f] \cap\cS_{n+1} \subset (0\,,f_n]
		\times\left(\e^n,\e^{n+1}\right]\qquad\text{for all integers $n\ge 0$}.
	\]
	Now, the condition \eqref{liminf:f:e} ensures that
	$f_n < \e^n$ for all sufficiently large integers $n\gg1$.
	Therefore,
	\[
		\text{\rm Epi}[f]\cap\cS_{n+1} \subseteq
		\bigcup_j R_{j,n},
	\]
	where the union is taken over all nonnegative integers
	$j\le \lfloor (\e^{n+1}-\e^n)/f_n \rfloor+1$,
	and
	\[
		R_{j,n} := ( 0\,,f_n]\times
		\left(\e^n + j f_n\,,\e^n + (j+1)f_n\right],
	\]
	for all integers $n,j\ge0$. Each $R_{j,n}$ is an upright square of
	side $f_n$. Therefore,
	\[
		\nu^n_\rho\left( \text{\rm Epi}[f]\right) \le
		\left(\left\lfloor\frac{\e^{n+1}-\e^n}{f_n}
		\right\rfloor +1\right)\cdot\left( \frac{f_n}{\e^n}\right)^\rho
		\le 2\e\left(\frac{f_n}{\e^n}\right)^{\rho-1},
	\]
	for all $\rho>0$ and $n\gg1$. It follows from this inequality that
	\[
		\Dimh\left(\text{\rm Epi}[f]\right) \le\inf\left\{\rho\ge 1:\
		\sum_{n=0}^\infty\left(\frac{f_n}{\e^n}\right)^{\rho-1}<\infty\right\}.
	\]
	Cauchy's test shows that $\sum_{n=0}^\infty (f_n/\e^n)^{\rho-1}$
	converges iff $\int_1^\infty[f^{-1}(x)/x]^{\rho-1} x^{-1}\d x$
	converges.  Therefore, the preceding display
	proves that $\Dimh(\text{\rm Epi}[f])$ is bounded from above by
	the expression on the right-hand side of \eqref{dimtest}. We would like to
	record the fact that  this
	part of the proof does not require $f$ to be convex.
	
	We now  derive a matching lower bound for
	$\Dimh(\text{\rm Epi}[f]) $. Define, for every integer $n\geq 0$,
	\begin{align*}
		A_n&:=(0\,, f_n]\times \left(\e^n, \e^{n+1}\right],\\
		A_n^u&:=\left\{(x\,,y)\in A_n:\  y>
			\frac{\e^{n+1}-\e^n}{f_n-f_{n-1}}(x-f_n)+\e^{n+1} \right\},\\
		A_n^l&:=\left\{(x\,,y)\in A_n:\  y\leq
			\frac{\e^{n+1}-\e^n}{f_n-f_{n-1}}(x-f_n)+\e^{n+1} \right\}.
	\end{align*}
	It might help to consider Figure \ref{fig:range(f)} at this point,
	keeping in mind that the rectangle $A_n$ is the disjoint union $A_n^u\cup A_n^l$.
	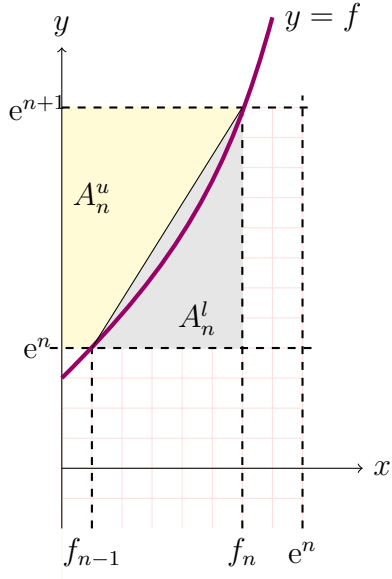
\begin{figure}[h!]\centering
	\begin{tikzpicture}[scale=.8,domain=-1.05:1.05]%\draw[help lines] (0,0) grid (4,8);
		\draw[step=.5,pink!50,very thin] (0,1) grid (4,8);
		\filldraw[yellow!20](0,4) to (0.5,4) to (3,8) to (0,8) to (0,0);
		\filldraw[gray!20](0.5,4) to (3,4) to (3,8);
		\draw[->] (0,2) -- (5,2) node[right] {\textcolor{black}{$x$}};
		\draw[->] (0,1) -- (0,9) node[above] {\textcolor{black}{$y$}};
		\draw[thick,dashed] (4,1) to (4,8.2);\draw[thick,dashed] (-0.2,4) to (4.2,4);
			\node at (4,0.6){$\e^n$};\node at (-0.4,4){$\e^n$};
		\draw[thick,dashed] (0,8) to (4.2,8);\draw[thick,dashed] (3,1) to (3,8);
			\node at (-0.4,8){$\e^{n+1}$};\node at (3,0.6){$f_n$};
		\draw[thick,dashed] (0.5,1) to (0.5,4);\node at (0.5,0.6){$f_{n-1}$};
		\draw[ultra thick,blue!40!red!100](0,3.5) to [bend right=15](3.5,9.5)%
			node[right]{\textcolor{black}{$y=f$}};
		\node at (0.5,6.5){$A^u_n$};\node at (2.25,4.5){$A^l_n$};
		\draw[thin](0.5,4)to(3,8);
	\end{tikzpicture}
	\caption{The range of $f$.}\label{fig:range(f)}
	\end{figure}
	
        Since $f$ is convex, $A_n^u \subseteq\text{\rm Epi}[f] \cap A_n$ for all $n\geq 0$.
        In addition, we have
        \[
        	\Dimh \left(\bigcup_{n=0}^\infty A_n^l  \right)
        	\leq \Dimh \left(\bigcup_{n=0}^\infty A_n^u  \right),
	\]
        thanks to mid-axial symmetry. Consequently,
        \begin{align*}
	        \Dimh\left(\text{\rm Epi}[f]\right)
	        	&\ge\Dimh\left(\bigcup_{n=0}^\infty A_n^u \right)\\
	        &= \max\left\{\Dimh\left(\bigcup_{n=0}^\infty A_n^u \right),
	        	\Dimh\left(\bigcup_{n=0}^\infty A_n^l \right) \right\}\\
		&=  \Dimh\left(\bigcup_{n=0}^\infty A_n \right).
        \end{align*}
        Therefore, it suffices to prove that
        \begin{equation}\label{goal:LB:Epi}\begin{split}
       		\Dimh\left(\bigcup_{n=0}^\infty A_n \right)&\geq \sup\left\{\rho\ge 1:\
			\sum_{n=0}^\infty\left(\frac{f_n}{\e^n}\right)^{\rho-1}
			=\infty\right\} \\
		&=\inf\left\{\rho\ge 1:\
			\sum_{n=0}^\infty\left(\frac{f_n}{\e^n}\right)^{\rho-1}<\infty\right\}.
	\end{split}\end{equation}
	This and Cauchy's integral test together prove that $\Dimh(\text{\rm Epi}[f])$
	is bounded from below by the right-hand side of \eqref{dimtest}, and hence complete
	the proof.
	
	Let us define a Borel measure $\mu$ on $\R^2$
	as follows: For every Borel set $E\subset \R^2$,
	\[
		\mu(E\cap \cS_n) :=\left| E\cap A_n\right|,
	\]
	where $|\cdots|$ denotes the 2-dimensional Lebesgue measure.
	For every upright box of the form $(x_1\,,x_1+r]\times (x_2\,, x_2+r]$---%
	where $r\geq 1$---and for every $\rho \in (0\,,1]$,
	\begin{align*}
		\mu\left( (x_1\,,x_1+r]\times (x_2\,, x_2+r]\right) &\leq
			\left(r \wedge f_n \right)\times r\\
		&= \left(\frac{r \wedge f_n}{f_n}\right) \times f_n\times r\\
		&\leq \left(\frac{r}{f_n}\right)^\rho \times f_n\times r\\
		&= r^{1+\rho} f_n^{1-\rho}.
	\end{align*}
	We now use the density theorem of Barlow and Taylor \cite[Theorem 4.1]{BT1992},
	and the fact that $\mu(A_n)=f_n ( \e^{n+1}-\e^n )$, in order to obtain
	the following:
	\[
		\nu_{1+\rho}^n(A_n) \geq  f_n^{\rho-1}\e^{-n(1+\rho)}\mu(A_n)
		\geq \text{const}\cdot \left(\frac{f_n}{\e^n}\right)^\rho.
	\]
	This inequality immediately implies \eqref{goal:LB:Epi}
	and completes the proof of the lower bound under the more restrictive
	condition \eqref{liminf:f:e}.
	
	To complete the argument, we outline how one changes the preceding to
	accommodate the more general condition \eqref{liminf:f}.
	Choose and fix a real number $a>1$ such that
	\[
		\liminf_{x\to\infty}\,\frac{f(x)}{x}>a,
	\]
	and redefine $\{\cV_k\}_{k\ge0}$---%
	hence also $\{\cS_k\}_{k\ge0}$---in \eqref{VS}. Now repeat the preceding
	argument but everywhere replace $\e^n$ and $\e^{n+1}$  by $a^n$ and
	$a^{n+1}$ respectively. Proposition \ref{pr:BT:Dim} ensures that these changes do not
	affect the end result of the method.
\end{proof}

Having just completed the proof of Proposition \ref{pr:Epi},
we can now establish Corollary \ref{co:Epi}. This proof will conclude
the material of this section.

\begin{proof}[Proof of Corollary \ref{co:Epi}]
	The function $f(x):=x(\log_+x)^p$ satisfies the conditions of
	Proposition \ref{pr:Epi}, and one can deduce the asserted formula
	for $\Dimh(\mathcal{E}_p)$ from Proposition \ref{pr:Epi}.
	
	If $q>1$, then $f(x):=x^q$ satisfies
	\eqref{liminf:f}, and Proposition \ref{pr:Epi} immediately shows that
	$\Dimh(\widetilde{\mathcal{E}_q})=1$.
	
	Finally, if $q\in[0\,,1]$, then it suffices to prove
	that $\Dimh(\widetilde{\mathcal{E}_q})\ge2$.
	Because $\widetilde{\mathcal{E}_q}\cap\cS_n$ contains
	\[
		\mathcal{F}_n:=\left\{(x\,,y):\, \e^n<
		x\le y\le \e^{n+1}\right\},
	\]
	for all sufficiently large integers $n\gg1$, it suffices to prove that
	\[
		\Dimh\left(\bigcup_{n=0}^\infty\mathcal{F}_n\right)\ge2.
	\]
	Define
	\[
		\mathcal{F}_n':=\left\{(x\,,y):\, \e^n< y\le x \le \e^{n+1}\right\},
	\]
	and note that
	\[
		\Dimh\left(\bigcup_{n=0}^\infty\mathcal{F}_n\right)=
		\Dimh\left(\bigcup_{n=0}^\infty\mathcal{F}_n'\right),
	\]
	thanks to mid-axial symmetry.
	As was observed first by Barlow and Taylor \cite{BT1992},
	\[
		\Dimh(A\cup B) = \max\{\Dimh(A)\,,\Dimh(B)\},
	\]
	for all
	sets $A$ and $B$. Therefore, it remains to prove that
	\begin{equation}\label{eq:DD}
		\Dimh\left(\bigcup_{n=0}^\infty[\mathcal{F}_n\cup\mathcal{F}_n']
		\right)\ge2.
	\end{equation}
	
	Since $\cup_{n=0}^\infty[\mathcal{F}_n\cup\mathcal{F}_n']=
	\cup_{n=0}^\infty(\e^n,\e^{n+1}]^2$, we can define a measure $\mu$
	on $\cup_{n=0}^\infty[\mathcal{F}_n\cup\mathcal{F}_n']$ as follows:
	\[
		\mu\left( A\cap\cS_n\right)
		:=\left| A\cap\left(\e^n,\e^{n+1}\right]^2\right|,
	\]
	for all $n\ge 0$ and Borel sets $A\subset\R^2_+$, where $|\,\cdots|$
	denotes the planar Lebesgue measure. Since
	$\mu((a\,,a+r]\times(b\,,b+r])\le r^2$ and
	$\mu(\cS_n)\ge \text{const}\cdot \e^{2n}$ uniformly for all $n\ge 1$,
	an appeal to a density theorem of Barlow and Taylor \cite[Theorem 4.1]{BT1992}
	yields \eqref{eq:DD}.
\end{proof}

\section{Proof of Theorem \ref{th:SHE} and Corollary \ref{co:SHE}}
\subsection{A large deviations estimate}

The following is the main result of this section.

\begin{proposition}\label{pr:tails}
	If \eqref{cond:interm} holds,
	then there exist positive and finite constants $b_0$,
	$K_0<L_0$, and $K<L$ such that
	\[
		K_0\exp\left( -L\beta^{3/2}t\right)\le
		\P\left\{ u(t\,,x)>\e^{\beta t}\right\} \le L_0\exp\left( -K\beta^{3/2} t\right),
	\]
	uniformly for all  $t>1$, $\beta> b_0$, and $x\in\R$.
\end{proposition}

The proof hinges on the following moment inequality that
was mentioned earlier in the Introduction.

\begin{lemma}[Joseph et al \protect{\cite{JKM}}]\label{lem:moments}
	There exist positive and finite constants $M_0<N_0$ and $M<N$ such that
	\eqref{eq:moments:SHE} holds
	uniformly for all real numbers $k\ge 2$, $t>0$, and $x\in\R$.
\end{lemma}

We will use Lemma \ref{lem:moments} in order to establish
Proposition \ref{pr:tails} in two steps: An upper bound (see
Lemma \ref{lem:LD:UB}) and a ``matching'' lower bound
(see Lemma \ref{lem:LD:LB}). Those results are presented
in the sequel, and without further comment.

\begin{lemma}\label{lem:LD:UB}
	Let $N$ denote the constant of Lemma \ref{lem:moments}. Then
	for all $\beta\ge 12 N$,
	\[
		\limsup_{t\to\infty} t^{-1}\sup_{x\in\R}\log
		\P\left\{ u(t\,,x)>\e^{\beta t}\right\}\le - \frac{2\beta^{3/2}}{3\sqrt{3N}}.
	\]
\end{lemma}

\begin{proof}
	Choose and fix a real number $k\ge 2$.
	By Lemma \ref{lem:moments} and Chebyshev's inequality,
	the following holds uniformly for all $t>0$ and $x\in\R$:
	\begin{equation}\label{LD:UB}
		t^{-1}\log\P\left\{ u(t\,,x)>\e^{\beta t}\right\} \le \frac{k\log N_0}{t}
		-\beta k + Nk^3 = -\beta k + Nk^3 + o(1),
	\end{equation}
	as $t\to\infty$. The asymptotically-optimum choice of $k$ is $\sqrt{\beta/3N}$,
	which is at least $2$ when $\beta\ge12N$. Plug  $k:=\sqrt{\beta/3N}$
	in \eqref{LD:UB} to obtain the lemma.
\end{proof}

\begin{lemma}\label{lem:LD:LB}
	Let $M$ and $N$ denote the constants of Lemma \ref{lem:moments}. Then
	for all $\beta\ge 4M$,
	\[
		\liminf_{t\to\infty} t^{-1}\inf_{x\in\R}\log
		\P\left\{ u(t\,,x)>\e^{\beta t}\right\}
		\ge -\left(\frac{\beta}{M}\right)^{3/2}\inf_{q>1}\frac{q(Nq^2-M)}{q-1}.
	\]
\end{lemma}

\begin{proof}
	Choose and fix real numbers $q>1$ and $m\ge 2$, and define
	\[
		C_{m,q} := \left( 1- 2^{-m}\right)^{q/(q-1)}.
	\]
	By the Paley--Zygmund inequality (apply \cite[Lemma 7.3, p.\ 64]{CBMS} with $n:=qm$),
	\[
		\P\left\{ u(t\,,x) >\tfrac12\|u(t\,,x)\|_m\right\} \ge
		C_{m,q}\cdot\frac{\left[\E\left(|u(t\,,x)|^m\right)\right]^{q/(q-1)}}{
		\left[\E\left(|u(t\,,x)|^{qm}\right)\right]^{1/(q-1)}}\ge
		C_{2,q}\cdot\frac{\left[\E\left(|u(t\,,x)|^m\right)\right]^{q/(q-1)}}{
		\left[\E\left(|u(t\,,x)|^{qm}\right)\right]^{1/(q-1)}},
	\]
	where $\| X\|_m:=[\E (X^m)]^{1/m}$ for every
	$X\in L^m(\Omega)$.
	Therefore, Lemma \ref{lem:moments} implies that
	\begin{equation}\label{PZ:LB}
		\P\left\{ u(t\,,x) >\tfrac12\|u(t\,,x)\|_m\right\} \ge
		D_{m,q}\cdot\exp\left( -\frac{qm^3t(Nq^2-M)}{q-1}\right),
	\end{equation}
	where $D_{m,q} := C_{2,q} (M_0/N_0)^{mq/(q-1)}$.
	
	If $4M\le \beta <m^2M$, then
	\begin{equation}\label{gogo}
		\e^{\beta t} < \tfrac{1}{2}M_0 \e^{Mm^2 t}\le \tfrac12\|u(t\,,x)\|_m,
	\end{equation}
	for all $x\in\R$, and $t$ sufficiently large. Therefore,
	\eqref{PZ:LB} and \eqref{gogo} together imply that
	\begin{equation}\label{zoo}
		\P\left\{ u(t\,,x) >\e^{\beta t}\right\} \ge
		D_{m,q}\cdot\exp\left( -\frac{qm^3t(Nq^2-M)}{q-1}\right),
	\end{equation}
	for all $m\ge2$ such that $4M \le \beta<m^2 M$.  In particular,
	\[
		\liminf_{t\to\infty} t^{-1}\log\P\left\{ u(t\,,x)>\e^{\beta t}\right\}
		\ge -\frac{qm^3(Nq^2-M)}{q-1},
	\]
	for all $m\ge2$ such that $4M\le \beta<m^2M$. Let $m$ tend
	downward to ${\sqrt{\beta/M}}$ in order
	to see that
	\[
		\liminf_{t\to\infty} t^{-1}\log\P\left\{ u(t\,,x)>\e^{\beta t}\right\}
		\ge -\frac{q(Nq^2-M)}{q-1}\cdot\left(\frac{\beta}{M}\right)^{3/2},
	\]
	for all $q>1$. We can optimize the right-hand side of this expression over all
	$q>1$ to complete the derivation.
\end{proof}

\begin{proof}[Proof of Proposition \ref{pr:tails}]
	The upper bound in the statement of the proposition follows
	immediately from Lemma \ref{lem:LD:UB}. In order to deduce the
	lower bound, we first apply \eqref{zoo} with
	$m:=\sqrt{2\beta/M}$ in order to see that uniformly
	for all $x\in\R$, $\beta\ge 4M$, $q>1$, and large $t>0$,
	\[
		\P\left\{ u(t\,,x) >\e^{\beta t}\right\} \ge
		C_{2,q}\left(\frac{M_0}{N_0}\right)^{\sqrt{2\beta/M}q/(q-1)}
		\cdot\exp\left( -\frac{q(2\beta)^{3/2}t(Nq^2-M)}{(q-1)M^{3/2}}\right).
	\]
	In particular,
	\[
		\liminf_{\beta\to\infty}
		\beta^{-3/2}\inf_{t>0}
		\left[t^{-1} \log\P\left\{ u(t\,,x) >\e^{\beta t}\right\}\right]  \ge
		-\left(\frac {2}{M}\right)^{3/2}\inf_{q>1}\left[\frac{q (Nq^2-M)}{(q-1)}\right]
		>-\infty.
	\]
	These facts and Lemma \ref{lem:LD:LB}
	together establish the lower bound of the
	proposition for $t$ sufficiently large, say $t>t_0$ for
	a sufficiently-large $t_0>1$. When
	$t\in(1\,,t_0)$, we appeal to \eqref{PZ:LB}, but adjust the constants
	in \eqref{zoo} suitably.
\end{proof}

\subsection{Correlation length}%\label{sec:corr}

The main result of this section is a carefully-stated
version of the following (see Theorem \ref{th:localization}): As $t\to\infty$,
the correlation length of $x\mapsto u(t\,,x)$ is at least
$at^2$ for a suitable constant $a$. The proof relies on a localization
idea that was introduced in Conus  et al \cite{CJK}.

First, recall \cite{DZ,Walsh} that the solution to \eqref{SHE} can be written
as the unique solution to the stochastic integral equation,
\begin{equation}\label{eq:mild}
	u(t\,,x)
	= (p_t*u_0)(x)+ \int_{(0,t)\times \R}
	p_{t-s}(y-x)\sigma\left( u(s\,,y)\right)
	\xi(\d s\,\d y),
\end{equation}
 valid for all $x\in\R$ and $t>0$, where $p_t(x)$ denotes the heat kernel; that is,
 \begin{equation}\label{heat:kernel}
 	p_t(x) := \frac{\e^{-x^2/(2t)}}{\sqrt{2\pi t}}\qquad[t>0,x\in\R].
 \end{equation}

Now, let us choose and fix some $c>0$, and define intervals
\[
	\mathcal{I}(x\,,t\,; c) := \left[ x- \sqrt{ct}\,,x+ \sqrt{ct}\right],
\]
for every $x\in\R$ and $t>0$.
Define $u^{(c)}_0(x):=u_0(x)$ for all $x\in\R$,
and consider the random integral equation,
\begin{equation}\label{eq:heat:local}
	u^{(c)}(t,x)
	= (p_t*u_0)(x)+ \int_{(0,t)\times \mathcal{I}(x,t;c)}
	p_{t-s}(y-x)\sigma\left( u^{(c)}(s\,,y)\right)
	\xi(\d s\,\d y),
\end{equation}
 for all $x\in\R$ and $t>0$. This is a ``localized form'' of the solution $u$
 to \eqref{SHE}.

 One can prove that \eqref{eq:heat:local}
 has a unique strong solution, in the usual way, using Picard's iteration;
 see Proposition \ref{pr:loc} for a statement.
 Since we will need to pay close attention to the quantitative details of
 the argument---see in particular \eqref{eq:loc} below---we
 work out the details of that argument [together with the requisite
 estimates] in this section.

 Let $u^{(c,0)}(t,x) := u_0(x)$ for all $t\ge0$ and $x\in\R$, and then
 define
 \begin{equation}\label{eq:Un}
	u^{(c,n+1)}(t\,,x)
	= (p_t*u_0)(x)+ \int_{(0,t)\times\mathcal{I}(x,t;c)}
	p_{t-s}(y-x)\sigma\left( u^{(c,n)}(s\,,y)\right)
	\xi(\d s\,\d y),
 \end{equation}
 iteratively for all $n\ge0$.

 The sequence $u^{(c,1)}, u^{(c,2)},\ldots$ is basically
 the Picard-iteration approximation to the
 desired solution $u^{(c)}$ of \eqref{eq:heat:local}. Our first lemma
 estimates the moments of each $u^{(c,n)}$.

 \begin{lemma}\label{lem:moments:loc}
 	There exist positive and finite constants $A,A_0$
	such that for
	all real numbers $k\ge 2$,
	$t\ge0$, and $x\in \R$,
 	\[
		\sup_{c>0}\sup_{n\ge 0}
		\E\left( \left| u^{(c,n)}(t\,,x) \right|^k\right) \le
		A_0^k \e^{A k^3 t}.
	\]
 \end{lemma}

 \begin{proof}
 	Since $p_t*u_0$ is bounded uniformly by $\|u_0\|_{L^\infty(\R)}$,
	Minkowski's inequality yields
 	\begin{align}\label{UB0}
		\left\| u^{(c,n+1)}(t\,,x) \right\|_k
			&\le\|u_0\|_{L^\infty(\R)} + \left(4k
			\int_0^t\d s\int_{\mathcal{I}(x,t;c)}\d y\
			\left[ p_{t-s}(y-x)\right]^2 \left\| \sigma\left(u^{(c,n)}(s\,,y)\right)
			\right\|_k^2\right)^{1/2}\\\notag
		&\le\|u_0\|_{L^\infty(\R)} + \left(4k
			\int_0^t\d s\int_{-\infty}^{\infty}\d y\
			\left[ p_{t-s}(y-x)\right]^2 \left\| \sigma\left(u^{(c,n)}(s\,,y)\right)
			\right\|_k^2\right)^{1/2},
	\end{align}
	where we have used a special form of the Burkholder--Davis--Gundy inequality
	\cite[Theorem B.1, p.\ 103]{CBMS}
	in the first inequality. Because $\sigma$ is Lipschitz and vanishes at zero,
	$|\sigma(z)|\le\lip|z|$  for all $z\in\R$,
	where $\lip:=\sup_{z\in\R/0} |\sigma(z)/z|$ denotes  the Lipschitz constant
	of $\sigma$. Therefore,
 	\begin{align}\label{UB1}
		\left\| u^{(c,n+1)}(t\,,x) \right\|_k
			&\le\|u_0\|_{L^\infty(\R)} + \lip\left(4k
			\int_0^t\d s\int_{-\infty}^\infty\d y\
			\left[ p_{t-s}(y-x)\right]^2 \left\| u^{(c,n)}(s\,,y)
			\right\|_k^2\right)^{1/2}\\\notag
		&\le\|u_0\|_{L^\infty(\R)} + \e^{\alpha t}
			\lip\mathcal{N}_{\alpha,k}\left(
			u^{(c,n)}\right)\cdot\left(4k\int_0^t \e^{-2\alpha s}
           \| p_s(\cdot)\|_{L^2(\R)}^2\,\d s \right)^{1/2},
	\end{align}
	where
	\begin{equation}\label{N}
		\mathcal{N}_{\alpha,k}(\Phi) := \sup_{t\ge0}\sup_{x\in\R}\left(
		\e^{- \alpha t}\|\Phi_t(x)\|_k\right),
	\end{equation}
	for all space-time random fields $\Phi$, and all real numbers
	$\alpha$ and $k\in[2\,,\infty)$. Since
	\[
		\int_0^\infty \e^{-2\alpha s}\| p_s(\cdot) \|_{L^2(\R)}^2\,\d s=
		\frac{\text{const}}{\sqrt\alpha},
	\]
	it follows from \eqref{UB1} that
 	\begin{equation}\label{Q}
		\mathcal{N}_{\alpha,k}\left( u^{(c,n+1)}\right)
		\le\|u_0\|_{L^\infty(\R)} + \frac{(Qk)^{1/2}}{\alpha^{1/4}}
		\mathcal{N}_{\alpha,k}\left(
		u^{(c,n)}\right),
	\end{equation}
	where $Q$ is a finite constant that does not depend on $(c\,,k\,,\alpha\,,n)$.
	In particular, we can set $\alpha:=16 (Qk)^2$ to see that
	\begin{equation}\label{UB2}
		\mathcal{N}_{16(Qk)^2,k}\left( u^{(c,n+1)}\right)
		\le\|u_0\|_{L^\infty(\R)} + \tfrac12\mathcal{N}_{16(Qk)^2,k}\left(
		u^{(c,n)}\right),
	\end{equation}
	for all integers $n\ge0$ and reals $k\ge2$. Because
	\begin{equation}\label{11}
		\mathcal{N}_{16(Qk)^2,k}\left( u^{(c,0)}\right) =\|u_0\|_{L^\infty(\R)},
	\end{equation}
	we iterate \eqref{UB2} in order to see that
	\begin{equation}\label{12}
		\mathcal{N}_{16(Qk)^2,k}\left( u^{(c,n)}\right) \le
		2\|u_0\|_{L^\infty(\R)}\qquad
		\text{for all $n\ge 0$}.
	\end{equation}
	Equivalently,
	\[
		\E\left( \left| u^{(c,n)}(t\,,x) \right|^k\right) \le
		\left( 2\|u_0\|_{L^\infty}\right)^k
		\e^{16 Q^2 k^3 t}.
	\]
	The lemma follows from this inequality.
 \end{proof}

 Next we wish to show that $\{u^{(c,n)}\}_{n=0}^\infty$ forms a Cauchy sequence.
 En route we will also control carefully the size of the gaps
 $u^{(c,n+1)}-u^{(c,n)}$ of that Cauchy sequence.

 \begin{lemma}\label{lem:moments:diff:loc}
 	Let $A$ be as in the statement of Lemma \ref{lem:moments:loc}.
	There exist positive and finite
	constants $B_0$ and $B$ such that for all integers $n\ge 1$ and real numbers
	$k\ge2$, and $t\ge0$,
 	\[
		\sup_{c>0}\sup_{x\in\R}
		\E\left( \left| u^{(c,n+1)}(t\,,x) - u^{(c,n)}(t\,,x)\right|^k\right)
		\le B_0^k \e^{Ak^3t - Bnk}.
	\]
 \end{lemma}

 \begin{proof}
 	As in \eqref{UB0}, we obtain
 	\begin{align*}
		&\left\| u^{(c,n+1)}(t\,,x) - u^{(c,n)}(t\,,x)\right\|_k^2\\
		&\hskip1in\le 4k\lip^2\cdot \int_0^t\d s\int_{-\infty}^\infty\d y\
			\left[ p_{t-s}(y-x)\right]^2\left\| u^{(c,n)}(s\,,y)
			- u^{(c,n-1)}(s\,,y)\right\|_k^2.
	\end{align*}
	Therefore, if we define $\mathcal{N}_{\alpha,k}$ as in \eqref{N}, then
	\begin{align*}
	\mathcal{N}_{\alpha,k}\left( u^{(c,n+1)} - u^{(c,n)}\right)
		&\le (4k)^{1/2}\lip\cdot\mathcal{N}_{\alpha,k}\left( u^{(c,n)}-
			u^{(c,n-1)}\right)\cdot
			\left(\int_0^t \e^{-2\alpha s}\|p_s(\cdot)\|_{L^2(\R)}^2\,\d s\right)^{1/2}\\
		&\le \frac{(Qk)^{1/2}}{\alpha^{1/4}}\cdot\mathcal{N}_{\alpha,k}\left( u^{(c,n)}-
			u^{(c,n-1)}\right),
	\end{align*}
	where $Q$ denotes the same constant that appeared earlier in \eqref{Q}.
	It follows that
	\[
		\mathcal{N}_{16(Qk)^2, k}\left( u^{(c,n+1)} - u^{(c,n)}\right)
		\le \tfrac12\mathcal{N}_{16(Qk)^2, k}\left( u^{(c,n)} - u^{(c,n-1)}\right),
	\]
	for all $n\ge 1$.
	Lemma \ref{lem:moments:loc} and its proof together show that
	both sides of the preceding inequality are finite. Therefore,
	iteration yelds the following for all $n\ge 1$:
	\[
		\mathcal{N}_{16(Qk)^2, k}\left( u^{(c,n+1)} - u^{(c,n)}\right)
		\le 2^{-n}\mathcal{N}_{16(Qk)^2, k}\left( u^{(c,1)} - u^{(c,0)}\right).
	\]
	This, \eqref{11}, and \eqref{12} together yield the following for all $n\ge 1$:
	\[
		\mathcal{N}_{16(Qk)^2, k}\left( u^{(c,n+1)} - u^{(c,n)}\right)
		\le \frac{3\|u_0\|_{L^\infty(\R)}}{2^n},
	\]
	which is more than enough to establish the lemma.
 \end{proof}

Lemmas \ref{lem:moments:loc} and \ref{lem:moments:diff:loc} readily
yield the following.

\begin{proposition}\label{pr:loc}
	The random integral equation \eqref{eq:heat:local} admits a
	predictable solution $u^{(c)}$ that is unique among all solutions that
	satisfy the inequality
	\[
		\sup_{c>0}
		\E\left( \left| u^{(c)}(t\,,x)\right|^k\right)\le A_0^k\e^{Ak^3t}
		\qquad\text{for all $t\ge0$, $x\in\R$, and $k\ge 2$},
	\]
	where $A$ and $A_0$ are as in Lemma \ref{lem:moments:loc}. In addition,
	there exists a finite constant $B_1$ such that
	\begin{equation}\label{eq:loc}
		\sup_{c>0}\sup_{x\in\R}
		\E\left( \left| u^{(c)}(t\,,x)- u^{(c,n)}(t\,,x) \right|^k\right)\le
		B_1^k \e^{Ak^3t-Bnk},
	\end{equation}
	for all $t\ge0$ and $k\ge 2$, $B$ was defined in Lemma \ref{lem:moments:diff:loc}.
\end{proposition}

As was implied earlier, it is a standard fact that the solution $u^{(c)}$
to \eqref{eq:heat:local} exists and is unique. The key feature of the preceding
is the quantitative bound \eqref{eq:loc}, which is a byproduct of our particular
method.

Next we prove that if $c$ is large enough then $u^{(c)}\approx u$. This is also
a natural statement. We are being careful only because we need to be able to control
the size of the error $u^{(c)}-u$. The following does that for us.

\begin{lemma}\label{lem:u-uc}
	There exists a positive and finite constant $C_0=C_0(N\,,N_0\,,\lip)$ such that
	\[
		\sup_{x\in\R}
		\E\left( \left| u(t\,,x) - u^{(c)}(t\,,x)\right|^k\right) \le C_0^k
		\e^{-(c- (64\lip^4 \vee 2N)k^2t)k/2},
	\]
	uniformly for all real numbers $c,t>0$ and $k\ge2$.
\end{lemma}

\begin{proof} 
	We begin by studying a slightly different problem.
	
	Let us combine \eqref{eq:heat:local}  and \eqref{eq:mild},
	via the Burkholder--Davis--Gundy inequality in the form given
	in \cite[Proposition 4.4, p.\ 36]{CBMS}  in order to see that for all $x\in\R$, $c,T>0$,
	$t\in(0\,,T)$, and $k\in[2\,,\infty)$,
	\begin{align*}
		&\left\| u(t\,,x) - u^{(c)}(t\,,x)\right\|_k^2\\
		&\hskip1in\le2\left\| \int_{(0,t)\times\left\{y\in\R:\, |y-x| \le \sqrt{ct}\right\}}
			p_{t-s}(y-x)\left[ \sigma\left( u(s\,,y)\right)-\sigma\left( u^{(c)}(s\,,y)\right)
			\right]\xi(\d s\,\d y)\right\|_k^2\\\notag
		&\hskip2in + 2\left\| \int_0^t\d s\int_{\left\{y\in\R:\, |y-x|> \sqrt{ct}\right\}}
			\d y\ p_{t-s}(y-x) \sigma(u(s\,,y))\,\xi(\d s\,\d y)
			\right\|_k^2\\\notag
		&\hskip1in\le 8k \int_0^t\d s\int_{\left\{y\in\R:\, |y-x|\le \sqrt{ct}\right\}}\d y\,
			\left[p_{t-s}(y-x)\right]^2
			\left\| \sigma\left( u(s\,,y)\right)-\sigma\left( u^{(c)}(s\,,y)\right)
			\right\|_k^2\\\notag
		&\hskip2in+ 8k\int_0^t\d s\int_{\left\{y\in\R:\, |y-x|> \sqrt{ct}\right\}}\d y\,
			\left[ p_{t-s}(y-x)\right]^2 \left\| \sigma(u(s\,,y))\right\|_k^2.
	\end{align*}
	We recall that $|\sigma(z)|\le\lip|z|$ and
	$|\sigma(w)-\sigma(z)|\le\lip|z-w|$ for all $w,z\in\R$. Thus,
	\begin{align*}
		&\left\| u(t\,,x) - u^{(c)}(t\,,x)\right\|_k^2\\
		&\hskip1in
			\le 8k\lip^2\int_0^t\d s\int_{\left\{y\in\R:\, |y-x|\le \sqrt{ct}\right\}}\d y\,
			\left[p_{t-s}(y-x)\right]^2
			\left\| u(s\,,y) - u^{(c)}(s\,,y) \right\|_k^2\\\notag
		&\hskip1.8in + 8k\lip^2\int_0^t\d s\int_{\left\{y\in\R:\, 
			|y-x|> \sqrt{ct}\right\}}\d y\,
			\left[ p_{t-s}(y-x)\right]^2 \left\| u(s\,,y)\right\|_k^2.
	\end{align*}
	Define, for all $0<s<t<T$,
	\begin{equation}\begin{split}
		M(s) &:= \sup_{y\in\R}\left\| u(s\,,y) - u^{(c)}(s\,,y) \right\|_k^2,\\
		P(s\,,t) &:= \int_{\left\{y\in\R:\, |y|> \sqrt{ct}\right\}}
			\left[p_{t-s}(y)\right]^2\,\d y.
		\label{eq:MP}
	\end{split}\end{equation}
	Lemma \ref{lem:moments} can now be used to control
	the size of $\|u(s\,,y)\|_k$ as follows:
	\begin{equation}\label{uPM}\begin{split}
              \left\| u(t\,,x) - u^{(c)}(t\,,x)\right\|_k^2
			&\le 8k\lip^2\int_0^t M(s) \,\d s
			\int_{\R}\d y\ [p_{t-s}(y-x)]^2 \\
		&\hskip1.3in + 8k\lip^2N_0^2 \int_0^t P(s\,,t) \e^{2Nk^2s}\, \d s.
	\end{split}\end{equation}
	
	Since the right-hand side of \eqref{uPM} is independent of $x$,
	we can rewrite \eqref{uPM} as the following self-referential
	inequality for the process $M$:
	\begin{align}\notag
		M(t) &\le 8 k \lip^2\int_0^t M(s)\,\d s\int_{\R} [p_{t-s}(y)]^2\,\d y
			+ 8k\lip^2N_0^2 \int_0^t P(s\,,t) \e^{2Nk^2s}\, \d s\\
		&=\frac{4 k \lip^2}{\sqrt{\pi}}
			\int_0^t \frac{M(s)}{\sqrt{t-s}}\,\d s
			+ 8k\lip^2N_0^2 \int_0^t P(s\,,t) \e^{2Nk^2s}\, \d s.
		\label{eq:MM}
	\end{align}
	The probability that a standard normal random variable exceeds $q>0$
	is at most $\frac12\exp(-q^2/2)$. Therefore, uniformly for all $t>s>0$,	\[
		P(s\,,t) =\frac{1}{\sqrt{\pi(t-s)}}\int_{\sqrt{2ct/[(t-s)]}}^\infty
		\frac{\e^{-z^2/2}}{\sqrt{2\pi}}\,\d z\le \frac{\e^{-c}}{
		2 \sqrt{\pi(t-s)}}.
	\]
	
	It follows from (\ref{eq:MM}) that for all $t>0$,
	\begin{equation}\label{MMN2}
		M(t) \le \frac{4k \lip^2}{\sqrt{\pi}}
		\int_0^t \frac{M(s)}{\sqrt{t-s}}\,\d s
		+\frac{4k\lip^2N_0^2 \,\e^{-c}} {\sqrt{\pi}}  
		\int_0^t  \frac{\e^{2Nk^2s}}{\sqrt{t-s}} \, \d s.
	\end{equation}
	
	Note that,  if $\Phi:\R_+\to\R$ is deterministic, then
	by \eqref{N},
	\[
		\mathcal{N}_{\eta}(\Phi) := 
		\mathcal{N}_{\eta,k}(\Phi) =
		\sup_{t>0} \left(\e^{-\eta t}\left | \Phi(t) \right |\right),
	\]
	for all $\eta > 0$ and $k\ge 2$. 
	With this in mind, it follows readily from
	the inequality \eqref{MMN2} that, uniformly for all $t>0$,
	\begin{align*}
		M(t) &\le \frac{4k \lip^2\,\e^{\eta t}}{\sqrt{\pi}}\,\mathcal{N}_{\eta}(M)
			\int_0^t \frac{\e^{-\eta(t-s)}}{\sqrt{t-s}}\,\d s
			+ \frac{4k\lip^2N_0^2\, \e^{-c+Nk^2 t}}{\sqrt{\pi}}  
			\int_0^t  \frac{\e^{-2Nk^2s}}{\sqrt{ s}} \, \d s.
	\end{align*}
	We multiply both sides by $\exp\{-\eta t\}$ and optimize over
	$t>0$ to find that 
	\[	
		\mathcal{N}_{\eta}(M) \le \frac{4k \lip^2}{\sqrt{\eta}}\,\mathcal{N}_{\eta}(M)
		+\frac{4\lip^2N_0^2} {\sqrt{2N}}  \,\e^{-c} \, \sup_{t>0}\e^{-(\eta-2Nk^2)t}.
	\]
	Lemma \ref{lem:moments} and Proposition \ref{pr:loc} together imply that
	\[
		K:=\sup_{t >0}\sup_{x\in\R}\E\left(\left|u(t\,,x)
		-u^{(c)}(t\,,x) \right|^k\right)<\infty.
	\]
	Consequently,
	$\mathcal{N}_{\eta}(M)\le K^{2/k}<\infty$
	regardless of the value of $\eta>0$.
	We now select $\eta:=\eta_{k}:= (64\lip^4 \vee 2N)k^2$ in order to find that	
	\[
		\mathcal{N}_{\eta_k}(M) \le \tfrac12 \mathcal{N}_{\eta_k}(M) +
		\frac{4\lip^2N_0^2\, \e^{-c}} {\sqrt{2N}},
	\]
	and hence,
	\[
		\mathcal{N}_{\eta_k}(M) \le \frac{8\lip^2N_0^2\, \e^{-c}} {\sqrt{2N}}.
	\]
	In particular, for all $t>0$,
	\[
		\sup_{y\in\R}\left\| u(t\,,y) - u^{(c)}(t\,,y) \right\|_k^2 = M(t)
		\le \frac{8\lip^2N_0^2} {\sqrt{\pi}}  \,\e^{- c + (64\lip^4 \vee2N)k^2t}.
	\]
	This yields the desired effect.
\end{proof}

We conclude this section with its main assertion, which is a carefully-phrased
version of the following statement: {\it If $t\to\infty$ and
if $|x-y|=\Omega(t^2)$ with a suitably-large constant, then
$u(t\,,x)$ and $u(t\,,y)$ are approximately independent.}
See Footnote 1 for the notation $f(t)=\Omega(g(t))$. 
Another way to say this
is that the ``correlation length'' of $x\mapsto u(t\,,x)$ is $\Omega(t^2)$.
It is believed that the said correlation length is $\Theta(t^{3/2})$;
to date, our $\Omega(t^2)$ bound is the best rigorously-known
lower bound for the correlation length.

The proof of the above assertion
is based on a coupling argument that is
similar to the fixed-time coupling of Conus et al \cite{CJK}, though
one has to compute the various constants very carefully in
the present large-time context.

First, let
\begin{equation}\label{M_1}
	M_1:=\max(2C_0\,,2B_1),
\end{equation}
where and $C_0$ and $B_1$
were defined respectively in
Lemma \ref{lem:u-uc} and Proposition \ref{pr:loc}. As was the case
with the various constants $C_0,B_1,N,N_0,\cdots$
that we have encountered so far,
the constant $M_1$
depends on various parameters of our SPDE, such as
$\lip$ and $\|u_0\|_{L^\infty(\R)}$, but is otherwise universal, once we
concentrate on a particular set of parameters in our stochastic heat equation
\eqref{SHE}.
	
\begin{theorem}\label{th:localization}
	There exists a finite constant $m_0>0$ such that
	for all $\mu,t\ge 1$ and $k\ge2$, and
	for all nonrandom points $x_1,\ldots,x_m\in\R$ that satisfy
	\begin{equation}\label{gap}
		\min_{1\le i\neq j\le m}|x_i - x_j| > m_0\mu^{3/2} t^2 k^3,
	\end{equation}
	there exist independent random variables $Y_1,\ldots,Y_m$
	such that for every integer $j=1,\ldots,m$:
	\begin{compactenum}[i{\rm )}]
	\item $Y_j$ depends only on $(t\,,x_j\,,k)$;
	\item $Y_j\in L^p(\Omega)$ for all real numbers $p\ge2$; and
	\item
		\(
			\E\left( \left| u(t\,,x_j) - Y_j\right|^k\right) \le 2M_1^k \e^{-\mu k^3t}.
		\)
	\end{compactenum}
\end{theorem}

\begin{proof}
	Proposition \ref{pr:loc} and Lemma \ref{lem:u-uc} together imply that
	\[
		\E\left( \left| u(t\,,x) - u^{(c,n)}(t\,,x)\right|^k\right)
		\le (2C_0)^k \e^{-ck/2+(32\lip^4\vee N)k^3t} + (2B_1)^k\e^{Ak^3t-Bnk},
	\]
	simultaneously for all real numbers $k\ge2$, $c,t>0$,
	$x\in\R$, and
	all integers $n\ge 0$.
	We can relabel $c\leftrightarrow 2ct $ and  set $n:=\lceil ct\rceil$ for constant $c>0$ 
	in order to see that
	\begin{equation}\label{eq:ineq}
		\sup_{x\in\R} \E\left( \left| u(t\,,x) - 
		u^{(2ct,\lceil ct\rceil)}_t(x)\right|^k\right)
		\le M_1^k \left[\e^{-kt\left[c-(32\lip^4 \vee N) k^2\right]} +
		\e^{-kt\left[Bc-Ak^2\right]}\right],
	\end{equation}
	where $M_1$ was defined in \eqref{M_1}.
	We apply this inequality with the following particular
	choice of the constant $c$:
	\[
		c := c_k := c(k\,,\mu) :=
		\left[ \mu +1+
		\max( 32 \lip^4\,, N\,,A) \right]
		\max\left\{ 1\,, B^{-1}\right\} k^2.
		%c := c_k := c(k\,,\mu) :=
		%\left[ \mu +1+ 8\lip^4 +
		%\max(N\,,A) \right]
		%\max\left\{ 1\,, B^{-1}\right\} k^2.
	\]
	We can backtrack through the constants $A,N$, and $B$ in order
	to see that $c$ is $k^2$ times a constant that depends only on $\mu$
	and the parameter choices of the SPDE \eqref{SHE}; it is otherwise
	universal. Moreover, this choice of $c$ is large enough to ensure that:
	\begin{compactenum}
	\item \eqref{eq:ineq} is applicable; and
	\item $\min(c- (32\lip^4 \vee N)k^2\,,Bc-Ak^2) > \mu k^2$.
	\end{compactenum}
	In particular, \eqref{eq:ineq} implies the following:
	\begin{equation}\label{key:est}
		\sup_{x\in\R} \E\left( \left| u(t\,,x) - u^{(2c_k t,\lceil c_kt\rceil)}_t(x)\right|^k\right)
		\le 2M_1^k \e^{-\mu k^3t},
	\end{equation}
	simultaneously for all real numbers $t>0$ and $k\ge2$.
	
	Recall that $u^{(2c_k t,0)}\equiv u_0$ is deterministic and
	\[
		u^{(2c_k t,1)}(t\,,x)
		= (p_t*u_0)(x) + \int_{(0,t)\times\left[x-t\sqrt{c_k}
		,x+t\sqrt{c_k}\right]}
		p_{t-s}(y-x)\sigma(u_0(y))\,\xi(\d s\,\d y);
	\]
	see also \eqref{eq:Un}. Choose and fix $t>0$, and observe that
	if $x_1,\ldots,x_m\in\R$ are points that
	satisfy $|x_i-x_j|>2t\sqrt{c_k}$ when $1\le i\ne j\le m$, then
	$u^{(2c_kt,1)}(t\,,x_1),\ldots,u^{(2c_kt,1)}(t\,,x_m)$ are independent.
	This is because the Wiener integrals
	$\int\varphi_1\,\d\xi$, \dots, $\int\varphi_m\,\d\xi$
	are independent if $\varphi_1,\ldots,\varphi_m\in L^2(\R_+\times\R)$
	have disjoint supports.
	
	Next we observe that $u^{(2c_kt,2)}(t\,,x_1),\ldots,u^{(2c_kt,2)}(t\,,x_m)$ are independent
	as long as  the points $x_1,\ldots,x_m\in\R$ satisfy $|x_i-x_j|>4t\sqrt{c_k}$
	when $1\le i\neq j\le m$. This is because: 
	\begin{compactenum}[(i)]
		\item With probability one, for every $x\in\R$,
			\[
				u^{(2c_kt,2)}(t\,,x)
				= (p_t*u_0)(x) + \int_{(0,t)\times\left[x-t\sqrt{c_k},x+t\sqrt{c_k}\right]}
				p_{t-s}(y-x)\sigma\left(u^{(2c_kt,1)}(s\,,y)\right)\xi(\d s\,\d y);
				\quad\text{and}
			\]
		\item If $\Phi_1,\ldots,\Phi_m$ are independent Walsh-integrable
			space-time random fields,
			then $\int_{E_1}\Phi_1\,\d\xi$, \dots, $\int_{E_m}\Phi_m\,\d\xi$ are independent
			as long as $E_1,\ldots,E_m$ are disjoint Borel subsets of $\R_+\times\R$.
	\end{compactenum}
	Indeed, (i) follows from the definition of $u^{(c,n)}$; and (ii) holds because
	$\int_{E_i}\Phi_i\,\d\xi$ and $\int_{E_j}\Phi_j\,\d\xi$ are uncorrelated
	jointly Gaussian random variables when $i\neq j$.
	
	An iteration of the preceding argument proves the following: For every
	integer $\ell\ge1$, the random variables
	$u^{(2c_kt,\ell)}(t\,,x_1),\ldots,u^{(2c_kt,\ell)}(t\,,x_m)$ are independent
	as long as
	\begin{equation}\label{eq:xx}
		|x_i-x_j|>2\ell t\sqrt{c_k}\qquad\text{when
		$1\le i\neq j\le m$.}
	\end{equation}
	We may apply this final
	observation with $\ell :=\lceil c_k t\rceil$, and then use \eqref{key:est}
	in order to deduce the theorem with
	\[
		Y_j := u^{(2c_kt,\lceil c_k t\rceil)}(t\,,x_j)\qquad\text{for $j=1,\ldots,m$.}
	\]
	This indeed concludes the theorem because \eqref{eq:xx}
	implies \eqref{gap}. 
	
	To prove the final assertion of the theorem,
	one may notice first that
	\[
		\lceil c_kt\rceil \le (c_k+1)(t+1)\le4c_kt,
	\]
	because $c_k$ and $t$ are both greater than one. Consequently,
	\begin{align*}
		2\ell t\sqrt{c_k} &= 2\lceil c_kt\rceil t\sqrt{c_k}\\
		&< 8t^2 c_k^{3/2}\\
		&<8 \mu^{3/2}t^2\left[ 2 + 
			\max(32\lip^4 \,, N\,,A) \right]^{3/2}
			\max\left\{ 1\,, B^{-3/2}\right\} k^3.
	\end{align*}
	Thus, one can set
	\[
		m_0 := 8\left[ 2 +  
		\max(32\lip^4 \,,N\,,A) \right]^{3/2}\max\left\{ 1\,, B^{-3/2}\right\},
	\]
	in order to see that indeed \eqref{eq:xx} implies \eqref{gap},
	which concludes the proof.
\end{proof}

\subsection{ A lower bound}

The main result of this section is the following lower bound
on the macroscopic Hausdorff dimension of the [$\vartheta$-rescaled]
space-time set of tall peaks of level $\beta$. We will prove it shortly.

\begin{theorem}\label{th:LB:main}
	For all $\beta>b_1:=\max\{\log 3\,,b_0/2\}$
	and $\vartheta\in(0\,,L^{-1}(2\beta)^{-3/2})$,
	the following holds with probability one:
	\[
		\Dimh\left\{ (x\,,t)\in\R\times(1\,,\infty):
		u(\vartheta\log t\,,x) > t^{\vartheta\beta}\right\}
		\ge 2-  L(2\beta)^{3/2}\vartheta.
	\]
\end{theorem}

Before we present the proof we pause and first establish a certain tail
probability inequality. That inequality will play a role in the proof of Theorem
\ref{th:LB:main}, which is presented afterward.

The following is the desired tail probability bound for the solution to
\eqref{SHE}.

\begin{proposition}\label{pr:smallball}
	For every real number $\beta>b_1$
	there exists a finite constant $m_1=m_1(\beta\,,L\,,M_1\,,K_0)$
	such that for all real numbers $t\ge1$, all integers $k \ge 2$ and $m\ge 3$,
	\begin{equation}\label{sb1}
		\P\left\{\max_{1\le j\le m} u(t\,,x_j) < \e^{\beta t}
		\right\} \le \exp\left( -\frac{K_0m}{2}\e^{-L(2\beta)^{3/2}t}\right)
		+\frac{K_0}{2}\exp\left(- k^3\, L(2\beta)^{3/2}t\right),
	\end{equation}
	uniformly for all points $x_1,\ldots,x_m\in\R$ that satisfy
	the gap condition
	\begin{equation}\label{gap1}
		\min_{1\le i\neq j\le m}|x_i-x_j| \ge m_1(\log m)^{3/2} t^2k^3.
	\end{equation}
\end{proposition}

\begin{proof}
%\begin{equation}\label{mu}
%\mu\in\left[\frac{L(2\beta)^{3/2}+\log\left(\dfrac{4mM_1^2}{K_0}\right)}{8}
%\,, L(2\beta)^{3/2}+\frac{\log\left(\dfrac{4mM_1^2}{K_0}\right)}{8}\right].
%\end{equation}
	Part iii) of Theorem \ref{th:localization} %[applied with $k=2$]
	and Chebyshev's inequality together imply
	that if the sequence
	$\{x_i\}_{i=1}^m$ satisfies the gap condition \eqref{gap},
	then there exist independent random variables $Y_1,\ldots,Y_m$
	such that
	\[\label{eq:u-Y:P}\begin{split}
		\P\left\{ \max_{1\le j\le m}|u(t\,,x_j) - Y_j| > z\right\}
			&\le \sum_{j=1}^m
			\P\left\{ |u(t\,,x_j)-Y_j|>z\right\}\\
		&\le \frac{2m M_1^k \e^{-k^3\mu t}}{z^k},
	\end{split}\]
	for all $z>0$, $\mu \ge 1$ and all integers $k \ge 2$. Because the 
    $Y_j$'s are independent, the preceding implies that
	\begin{equation}\label{maxu<Y}
    \begin{split}
		\P\left\{ \max_{1\le j\le m} u(t\,,x_j) < \e^{\beta t}\right\}
			&\le  \P\left\{ \max_{1\le j\le m}Y_j< 2\e^{\beta t}
			\right\} + \P\left\{ \max_{1\le j\le m}
			|u(t\,,x_j)-Y_j| > \e^{\beta t}\right\}\\
		&\le \prod_{j=1}^m\P\left\{ Y_j< 2\e^{\beta t}
			\right\} + 2m M_1^k \e^{-k(\beta + k^2\mu)t}\\
		&\le \prod_{j=1}^m\P\left\{ Y_j< 2\e^{\beta t}
			\right\} + 2mM_1^k\e^{-k^3\mu  t},
	\end{split}\end{equation}
	since $\beta>0$ and $t\ge 1$.
	Now, for every $j=1,\ldots,m$ and $n\ge2$,
	\begin{align*}
		\P\left\{  Y_j < 2\e^{\beta t}\right\}
			&\le  1 - \P\left\{ u(t\,,x_j) \ge 3\e^{\beta t}
			\right\} + 2mM_1^k \e^{-k(\beta + k^2\mu)t}\\
		&\le1 - \P\left\{ u(t\,,x_j) \ge \e^{2\beta t}
			\right\} + 2mM_1^k \e^{-k^3\mu t};
	\end{align*}
	the last line uses also the fact that $\beta> \log 3$
	and $t\ge1$ in order
	to deduce that $3\exp(\beta t)\le\exp(2\beta t)$.
	Because $2\beta>b_0$,  Proposition \ref{pr:tails} yields
	\begin{equation}\label{sb2}
		\P\left\{  Y_j < 2\e^{\beta t}\right\}
		\le1 - K_0\exp\left( -L(2\beta)^{3/2}t\right) +
		2mM_1^k\e^{-k^3\mu t},
	\end{equation}
	valid for all $j=1,\ldots,m$. Now we choose $\mu\ge 1$ such that
\[
\mu \ge L(2\beta)^{3/2} + \frac 1 {k^3 t}\log\Big(\dfrac{4m M_1^k}{K_0}\Big).
\]
Then for all integers $k \ge 2$, we have     
	\begin{equation}\label{mumu}
		2mM_1^k\e^{-k^3\mu t} \le
		\frac{K_0}{2}\exp\left( - k^3 L(2\beta)^{3/2}t\right).
	\end{equation}
This and (\ref{maxu<Y}) give the last term in (\ref{sb1}). Moreover,
it follows from (\ref{mumu}) and (\ref{sb2}) that  
	\begin{equation}\label{eq:prod:Y}\begin{split}
		\prod_{j=1}^m\P\left\{  Y_j < 2\e^{\beta t}\right\}
			&\le \prod_{j=1}^m
			\left( 1 - \frac{K_0}{2}\exp\left( -L(2\beta)^{3/2}t\right)\right)\\
		&\le\exp\left( -\frac{K_0m}{2}\e^{-L(2\beta)^{3/2}t}\right),
	\end{split}\end{equation}
	owing to the elementary real-variable inequality,
	$1-x\le\exp(-x)$
	valid for all $x>0$.
	Plug \eqref{eq:prod:Y} into \eqref{maxu<Y} and appeal to \eqref{mumu} once
	again to deduce that the probability inequality of the proposition is
	valid uniformly for every sequence $\{x_i\}_{i=1}^m$ that satisfies the minimum gap
	condition \eqref{gap}. To conclude, we need to verify that \eqref{gap1}
	implies \eqref{gap} for a suitable choice of $m_1$ which depends only
	on the initially-set parameters of \eqref{SHE}. If $m$ is sufficiently large%
	---how large depends only on $(\beta\,,L\,,M_1\,,K_0)$---then this is clear because
	$\mu/\ln m$ is bounded above and below by universal constants
	that depend only on $(\beta,,L\,,M_1\,,K_0)$ in that case. And if $m$ is below the threshold
	of being sufficiently large, then the proposition is tautologically true,
	uniformly for any choice of $m_1$.
\end{proof}

Now we have Proposition \ref{pr:smallball}, we conclude this section
with the following.

\begin{proof}[Proof of Theorem \ref{th:LB:main}]
	Let us
	choose the parameters $(\beta\,,\vartheta)$ as has been stated in the theorem,
	and consider the random set
	\[
		\mathcal{G} := \left\{ (x\,,t)\in\R\times(1\,,\infty):
		u(\vartheta\log t\,,x) > t^{\vartheta\beta}\right\},
	\]
	whose macroscopic dimension is of interest to us.
	Define
	\begin{equation}\label{G0}
		\mathcal{G}_0 := \mathcal{G} \cap \bigcup_{n=0}^\infty
		\left( \e^n,\e^{n+1}\right]^2.
	\end{equation}
	Since $\mathcal{G}_0\subseteq\mathcal{G}$ a.s., it suffices to prove
	that
	\begin{equation}\label{DimA_0}
		\Dimh(\mathcal{G}_0) \ge 2 - L(2\beta)^{3/2}\vartheta\qquad
		\text{a.s.}
	\end{equation}
	
	Since $\vartheta L(2\beta)^{3/2}<1$, we may choose and fix
	an arbitrary number
	\begin{equation}\label{gam}
		\gamma\in\left( L(2\beta)^{3/2}\vartheta\,,1\right).
	\end{equation}
	Define, for all integers $n\ge 0$,
	\[
		a_{j,n} := \e^n + j\e^{n\gamma}
		\qquad [0\le j<\e^{n(1-\gamma)}],
	\]
	and
	\[
		\mathcal{I}_{j,n}(\gamma) :=
		\left( a_{j,n}\,, a_{j+1,n}\right]
		\qquad [0\le j<\e^{n(1-\gamma)}].
	\]
	For all such integers $n$
	and $j$, and for every
	\begin{equation}\label{vareps}
		\varepsilon\in\left( 0 \,,\gamma-L(2\beta)^{3/2}\vartheta\right),
	\end{equation}
	we can find points
	$x_1,\ldots,x_m\in \mathcal{I}_{j,n}(\gamma)$---depending
	only on $(n\,,\gamma\,,\varepsilon\,,j)$---such that whenever $n>1$:
	\begin{compactenum}
		\item[\bf (C.1)] $|x_i-x_l|\ge \e^{n\varepsilon}$
			whenever $1\le i< l\le m$;
		\item[\bf (C.2)]  $\frac12\e^{n(\gamma-\varepsilon)} \le m \le
			2\e^{n(\gamma-\varepsilon)}$.
	\end{compactenum}
	Note in particular, that for every constant $Q$ there exists an integer
	$N(Q)\ge1$ such that
	\begin{equation}\label{N(Q)}
		\min_{1\le i\neq l\le m} |x_i-x_l| \ge
		\e^{n\varepsilon} \ge Q n^{7/2}\qquad\text{for all $n\ge N(Q)$}.
	\end{equation}
	Let $Q:=4m_1\vartheta^2 (\log 2+\gamma)^{3/2}k^3$, where $k\ge 2$ is an integer 
    such that 
 \begin{equation}\label{k}
 k^3 L(2\beta)^{3/2}\vartheta > 2 - \gamma.
 \end{equation}
 Observe that
	\[
		m_1(\log m)^{3/2}
		\left|\vartheta\log t\right|^2 k^3 \le
		m_1\vartheta^2\left(\log 2 + n(\gamma-\varepsilon)\right)^{3/2}
		(n+1)^2 k^3
		\le Q n^{7/2},
	\]
	uniformly for all reals $t\in(\exp(n)\,,\exp(n+1)]$ and integers $n\ge 1$.
	Therefore, the minimum gap condition \eqref{gap1} is valid with
	$t$ replaced by $\vartheta\log t$,
	when $t\in(\exp(n)\,,\exp(n+1)]$ and $n\gg1$ is sufficiently large, and 
    the integer $k$ in (\ref{k}). 
	In this way we see that Proposition \ref{pr:smallball} applies
	to yield the following: Uniformly 	for all  sufficiently-large 
    nonnegative integers  $n\gg1$, and for all real numbers
	$\beta>\max\{\log 3\,,b_0/2\}$, $\varepsilon\in(0\,,\gamma)$,
	and $t\in(\exp(n)\,,\exp(n+1)]$,
	\begin{align}\notag
		&\max_{\substack{j\in\Z\\
			0\le j<\e^{n(1-\gamma)}}}
			\P\left\{ \sup_{x\in \mathcal{I}_{j,n}(\gamma)}u(
			\vartheta \log t\,,x) < t^{\vartheta\beta}\right\}\\
		&\hskip2in\le  \exp\left( -\frac{K_0m}{2}\,
			\e^{-L(2\beta)^{3/2}\vartheta\log t}\right)
			+\frac{K_0}{2}\e^{-k^3 L(2\beta)^{3/2}\vartheta\log t} \notag \\
		&\hskip2in\le \exp\left( -\frac {K_0}4\,
			\e^{\kappa n}\right)
			+\frac{K_0}{2}\e^{-k^3 L(2\beta)^{3/2}\vartheta n},
			\label{explain}
	\end{align}
	with
	\begin{equation}\label{kappa}
		\kappa := \gamma-\varepsilon -L(2\beta)^{3/2}\vartheta.
	\end{equation}
	
	It is possible to explain the meaning of the bound
	\eqref{explain} in terms of
	the random set $\mathcal{G}_0$
	(see \eqref{G0}) as follows:
	\begin{align} \notag
		&\P\left\{ \mathcal{G}_0 \cap
			\left[   \mathcal{I}_{j,n}(\gamma) \times \{t\} \right] =\varnothing
			\text{ for some $0\le j< \e^{n(1-\gamma)}$ and
			$t\in\Z\cap\left(\e^n,\e^{n+1}\right]$}\right\}\\ \notag
		&\hskip1in\le\P\left\{ \min_{\substack{t\in\Z\\t\in(\e^n,\e^{n+1}]}}
			\min_{\substack{j\in\Z\\ \notag
			0\le j<\e^{n(1-\gamma)}}}
			\sup_{x\in \mathcal{I}_{j,n}(\gamma)}
			\frac{u(\vartheta\log t\,,x)}{t^{\vartheta\beta}}<1\right\}\\
		&\hskip1in\le\e^{n(2-\gamma) + 1}\left[\exp\left( 
			-\frac{K_0}4\, \e^{\kappa n}\right)
			+\frac{K_0}{2}\,
			\e^{- k^3L(2\beta)^{3/2}\vartheta n}\right], \label{explain2}
	\end{align}
	for all sufficiently large integers $n\gg1$.
	
	Of course, $\kappa>0$ and $k^3L(2\beta)^{3/2}\vartheta> 2 - \gamma$;
	thanks to \eqref{gam}, \eqref{vareps}, \eqref{k}, and \eqref{kappa}.
	Therefore, \eqref{explain2} and the Borel--Cantelli lemma
	together imply that
	\begin{equation}\label{a.s.:LB}
		\mathcal{G}_0 \cap \left[ \mathcal{I}_{j,n}(\gamma) \times \{t\} \right]
		\neq\varnothing
		\text{ for all $0\le j< \e^{n(1-\gamma)}$ and
		$t\in\Z\cap\left(\e^n,\e^{n+1}\right]$,}
	\end{equation}
	for all but a finite number of integers $n\ge 1$.

	Define
	\[
		\eta_{j,n}(t) := \inf\left\{ x\in\mathcal{I}_{j,n}(\gamma):\
		(x\,,t)\in\mathcal{G}_0\right\},
	\]
	where $\inf\varnothing:=\infty$. The assertion \eqref{a.s.:LB}
	ensures that $\eta_{j,n}(t)$ is a well-defined, finite, random variable for
	all integers $t\in(\e^n,\e^{n+1}]$,
	$0\le j<\e^{n(1-\gamma)}$, and $n\gg1$ sufficiently large.
	In fact, \eqref{a.s.:LB} can be stated in the following equivalent
	form: \\
	
	\noindent\textbf{Conclusion A.} {\it With probability one,
	$(\eta_{j,n}(s)\,,s)\in\mathcal{G}_0$ for all integers
	$s\in(\e^n,\e^{n+1}]$,
	$0\le j<\e^{n(1-\gamma)}$, and $n\gg1$ sufficiently large.}\\
	
	Now let $\mu$ denote a purely atomic random measure on $\R\times(1\,,\infty)$
	that is defined shell-by-shell as follows:
	For all Borel sets $A\subseteq\R\times(1\,,\infty)$ and $n\ge 0$,
	\[
		\mu(A\cap\cS_n) := \sum_{\substack{s\in\Z:\\
		\e^n<s\le \e^{n+1}}}\sum_{0\le j<\e^{n(1-\gamma)}}
		\1_A\left( \eta_{j,n}(s)\,,s\right),
	\]
	where $\1_A$ denotes the customary indicator function of the space-time set $A$,
	and $\1_A(\infty\,,s):=0$ for all $s\ge 0$ and $A\subseteq\R^2$.
	
	The $\mu$-mass of $\cS_n$ is easy to compute when $n$ is large. Indeed,
	\eqref{a.s.:LB}---see especially Conclusion A---ensures that with probability one,
	\begin{equation}\label{mass:mu:Sn}
		\mu(\cS_n) = \sum_{\substack{s\in\Z:\\
		\e^n<s\le \e^{n+1}}}\sum_{0\le j<\e^{n(1-\gamma)}}
		1 \ge \tfrac12 \e^{(2-\gamma)n},
	\end{equation}
	for all $n\gg1$ sufficiently large.
	
	Choose and fix an integer $n\ge 1$.
	Next consider an arbitrary integer $s\in(\e^n,\e^{n+1}]$, and
	real numbers $x$ and $r\ge 1$ such that $(x\,,x+r]\subseteq
	(\e^n,\e^{n+1}]$. We consider separately two cases:
	\begin{compactenum}
		\item[(a)] If $r\le \e^{\gamma n}$, then there are at most two
			integers $j_1$ and $j_2$ such that $\eta_{j_1}(s)$ and $\eta_{j_2}(s)$
			are in $(x\,,x+r]$. Consequently,
			$\mu ( (x\,,x+r]\times\{s\} ) \le 2
			\le 2r^\rho$ for every $\rho>0$. Sum this inequality over all
			$s\in(t\,,t+r]$ to see that
			\begin{equation}\label{ineq:1}
				\mu\left((x\,,x+r]\times (t\,,t+r]\right) \le 2r^{1+\rho},
			\end{equation}
			as long as $(x\,,x+r]\times(t\,,t+r]\subseteq\cS_n$,
			$1\le r\le 2\e^{\gamma n}$, and $\rho>0$.
		\item[(b)] If $\e^{\gamma n}<r\le \e^{n+1}-\e^n$, then
			every interval of the form $(x\,,x+r]\subseteq(\e^n,\e^{n+1}]$
			can contain at most $1+r\e^{-\gamma n}\le 2 r\e^{-\gamma n}$
			many points of the form $\eta_{j,n}(s)$. That is, in this case,
			for every integer $s\in(\e^n,\e^{n+1}]$ and all
			intervals $(x\,,x+r]\subseteq(\e^n,\e^{n+1}]$,
			\[
				\mu\left((x\,,x+r]\times \{s\}\right) \le 2r\e^{-\gamma n}
				\le 2\e^{-\gamma n} r^\rho \sup_{\e^{\gamma n}<r
				\le \e^{n+1}-\e^n} r^{1-\rho}
				\le2(\e-1)^{1-\rho} \e^{(1-\rho-\gamma) n} r^\rho,
			\]
			regardless of the value of $\rho>0$. We sum the preceding over
			all integers $s\in(t\,,t+r]$, and set $\rho:=1-\gamma$,
			in order to see that, as long as
			$(x\,,x+r]\times(t\,,t+r]\subseteq\cS_n$,
			\begin{equation}\label{ineq:2}
				\mu\left((x\,,x+r]\times (t\,,t+r]\right)
				\le 2(\e-1)^\gamma r^{2-\gamma}.
			\end{equation}
	\end{compactenum}
	
	At this stage, we can combine our observations \eqref{ineq:1} and 
    \eqref{ineq:2} in order to see that \eqref{ineq:2} in fact holds
	in both cases, as long as $r\ge 1$, 	$(x\,,x+r]\times(t\,,t+r]\subseteq\cS_n$,
	and $n\gg1$ is sufficiently large. Because $\mu$ is a measure on $\mathcal{G}_0$,
	the preceding fact and the density theorem of
	Barlow and Taylor \cite[Theorem 4.1]{BT1992} together imply that
	\[
		\nu^n_{2-\gamma}\left(\mathcal{G}_0\right) \ge \text{const}\cdot
		\e^{-(2-\gamma)n}\mu(\cS_n)
		\ge \text{const},
	\]
	a.s.\ uniformly for all $n\gg1$ large, where the last line is deduced from
	\eqref{mass:mu:Sn}. Consequently,
	\[
		\sum_{n=1}^\infty\nu^n_{2-\gamma}(\mathcal{G}_0)=\infty
		\quad\text{a.s.\ as long as $(\gamma\,,\varepsilon)$ satisfy
		\eqref{gam} and \eqref{vareps}}.
	\]
	Let $\varepsilon\downarrow0$ and $\gamma\downarrow L(2\beta)^{3/2}\vartheta$,
	without violating either \eqref{gam} or \eqref{vareps},
	in order to deduce half of the assertion \eqref{DimA_0}, namely that
	\[
		\Dimh(\mathcal{G}_0)
		\ge 2-  L(2\beta)^{3/2}\vartheta\qquad\text{a.s.}
	\]
    This proves \eqref{DimA_0}, thus concludes the proof of the theorem.
\end{proof}

We end this section with the following remark on cases of 
$\vartheta,\beta>0$ that may not be covered by the conditions of 
Theorem \ref{th:LB:main}.

\begin{remark}
	Theorem 3.10 has been formulated in a way that does not need 
	the following peculiar property. Thus, we include it here
	as a remark: For all $\vartheta,\beta>0$,
	\[
		\Dimh\left\{ (x\,,t)\in\R\times(1\,,\infty):
		u(\vartheta\log t\,,x) > t^{\vartheta\beta}\right\}
		\ge 1\qquad\text{a.s.}
	\]
	Indeed, the following stronger statement is valid:
	\[
		\Dimh\left( \{x\in\R:\ u(\vartheta\,,x) > \e^{\vartheta\beta}\}
		\times\{\e\}\right)= 1\qquad\text{a.s.}
	\]
	We will now prove the following equivalent formulation:
	\begin{equation}\label{ge1}
		\Dimh\left( \left\{|x|\ge m:\ u(\vartheta\,,x) > \e^{\vartheta\beta}\right\}
		\times\{\e\}\right)= 1\qquad\text{a.s.\ for all $m\ge\e$.}
	\end{equation}
	
	For all $m\ge\e$ and $\gamma>[\log m]^{-2/3}\vartheta\beta$, let
	\[
		\mathcal{H}_m:=\left\{|x|\ge m:\
		u(\vartheta\,,x) > \e^{\vartheta\beta}\right\}
		\supset\left\{ |x|\ge m:\ u(\vartheta\,,x) >\exp\left( \gamma\left[
		\log|x|\right]^{2/3}\right) \right\}.
	\]
	Let $\mathcal{R}_m(\gamma)$ denote the set on the right-hand side of the
	above expression. Since the set difference between $\mathcal{R}_m(\gamma)$
	and $\mathcal{R}_0(\gamma)$ is a bounded set, $\mathcal{R}_m(\gamma)\times\{\e\}$
	and $\mathcal{R}_0(\gamma)\times\{\e\}$ have the same [macroscopic]
	Hausdorff
	dimension. Therefore, the theory of Khoshnevisan, Kim, and Xiao
	\cite{KKX} implies that there exists a finite and positive constant
	$A$---independent of $(\gamma\,,m)$---such that with probability one,
	\[
		\Dimh\left( \mathcal{H}_m\times\{\e\}\right) \ge
		\sup_{\substack{\gamma>\vartheta\beta/(\log m)^{2/3}:\\
		\gamma\in\Q}}
		\Dimh\left( \mathcal{R}_0(\gamma)\times\{\e\}\right)
		\ge 1-\frac{A(\vartheta\beta)^{3/2}}{\log m}\qquad
		\text{for all $m\ge\e$}.
	\]
	The left-most quantity decreases as $m$ increases, whereas the right-most
	term increases with $m$. This proves \eqref{ge1}. 
%thus concludes the proof of the theorem.
\end{remark}

\subsection{An upper bound}

Theorem \ref{th:LB:main} implies the first inequality of Theorem \ref{th:SHE};
that is the lower bound on the Hausdorff dimension of $S_\vartheta(\mathscr{P}(\beta))$.
Now we work toward proving a complementary upper bound for the
Hausdorff dimension of the same sort of set. This effort begins with the
following technical lemma.

\begin{lemma}\label{lem:moments:sup}
	There exist positive and finite constants $N_1$ and $N_2$ such that
	\[
		\E\left( \sup_{x\in(a,a+1]}\sup_{t\in(b,b+1]}
		|u(t\,,x)|^k\right) \le N_1^k \e^{N_2k^3(b+1)},
	\]
	for every real number $a\in\R$, $b>0$, and $k\ge 2$.
\end{lemma}

In order to understand what this lemma says, let us note the following
formulation of \eqref{eq:moments:SHE}, which was mentioned already
in the Introduction:
\[
	\sup_{x\in(a,a+1]}\sup_{t\in(b,b+1]}\E\left(
	|u(t\,,x)|^k\right) \le N_0^k \e^{N k^3(b+1)}.
\]
Thus, Lemma \ref{lem:moments:sup} asserts that,
at cost of having slightly larger constants, we can ``put both of the
suprema inside the expectation.''

\begin{proof}[Proof of Lemma \ref{lem:moments:sup}]
	Throughout this proof we define
	\[
		\varrho(w\,,r) := |w|^{1/2} + |r|^{1/4},
	\]
	for all $w,r\in\R$. We may think of $\varrho$ as a
	``parabolic metric'' on space-time $\{(t\,,x):\ t\ge0,\, x\in\R\}$,
	where ``parabolic'' loosely refers to a kind of compatibility with the
	geometric structure of the heat equation.
	
	It is known that there exist finite constants $L_1$
	and $L_2$---independently of $(a\,,b)$---such that
	for all real numbers $k\ge2$,
	\[
		\sup_{0\le s\neq t\le b+1}\sup_{-\infty<x\neq y<\infty}
		\E\left( \left| \frac{u(s\,,y) - u(t\,,x)}{\varrho(y-x\,,s-t)}
		\right|^k\right) \le L_1^k  \e^{L_2 k^3(b+1)}:=C_{k,b}^k.
	\]
	See the proof of Theorem 1.3 of Conus, Joseph, and Khoshnevisan
	\cite{CJK}, for example.
	
	Define, for all $z\in\R$,
	\[
		\Lambda (z) :=\iint_{((a,a+1]\times(b,b+1])^2}
		\frac{\d x\,\d t\,\d y\,\d s}{[\varrho(y-x\,,s-t)]^z}.
	\]
	Then a quantitative form of the Kolmogorov continuity theorem
	\cite[Theorem C.6, p.\ 114]{CBMS}
	implies that for all real numbers $k\ge 2$,
	$q\in(0\,,1-(6/k))$, and $\delta\in(q\,,1-(6/k))$,
	\[
		\E\left( \sup_{\substack{x,y\in (a,a+1]\\x\neq y}}
		\sup_{\substack{s,t\in (b,b+1]\\s\neq t}} \left|
		\frac{u(s\,,y) - u(t\,,x)}{[\varrho(y-x\,,s-t)]^q}\right|^k\right)
		\le\frac{D^kC_{k,b}^k}{\delta^k(\delta-q)}
		\Lambda (12-k+k\delta),
	\]
	where $D$ is a finite constant that does not depend on $(a\,,b\,,k)$.
	Note that, as long as $k(1-\delta)>12$,
	\[
		\Lambda(12-k+k\delta)
		\le 4\int_{(0,1]^2}
		[\varrho(x\,,t)]^{k(1-\delta)-12}\,\d x\,\d t\\
		\le 4 \times 2^{k(1-\delta)-12} \le 2^k,
	\]
	since $\rho(x\,,t)\le \rho(1\,,1)=2$ for all $(x\,,t)\in(0\,,1]^2$.
	Therefore, Lemma \ref{lem:moments} ensures that for all $k>12/(1-\delta)$,
	\begin{align*}
		\E\left( \sup_{x\in(a,a+1]}\sup_{t\in(b,b+1]}
			|u(t\,,x)|^k\right)
			&\le 2^k\E\left( \sup_{x\in(a,a+1]}\sup_{t\in(b,b+1]}
			|u(t\,,x) - u_b(a)|^k\right) + 2^k\E\left(|u_b(a)|^k\right)\\
		&\le\frac{4^k D^kC_{k,b}^k}{\delta^k(\delta-q)} + (2N_0)^k\e^{Nk^3b}.
	\end{align*}
	This proves the lemma in the case that $k>12/(1-\delta)$. The conclusion 
    of Lemma \ref{lem:moments:sup}, in the
	case that $k\in[2\,,12/(1-\delta))$, follows from Jensen's inequality
	and the lemma in the case that $k>12/(1-\delta)$.
\end{proof}

Next, we present a ready consequence of Lemma \ref{lem:moments:sup}.
It might also help to recall the constants $N_1$ and $N_2$ from 
Lemma \ref{lem:moments:sup}.

\begin{proposition}\label{pr:sup:tail}
	There exist positive and finite constants $N_3$---depending
	only on $N_1, N_2, \beta$---and $N_4$---%
	depending only on $N_2$---such that for every  $\beta\ge 24N_2$
	and for all real numbers $a>0$ and $b\ge 1$
	and integers $\ell_1,\ell_2\ge 1$,
	\[
		\P\left\{ \exists t\in(b\,,b+\ell_2]:\
		\sup_{x\in(a,a+\ell_1]}u(t\,,x) > \e^{\beta t}\right\}
		\le N_3\ell_1\ell_2\exp\left( - N_4\beta^{3/2}b\right).
	\]
\end{proposition}

Proposition \ref{pr:sup:tail} is a ``maximal inequality'' that
corresponds to the pointwise inequality of Proposition \ref{pr:tails}.

\begin{proof}[Proof of Proposition \ref{pr:sup:tail}]
	Let $i$ and $j$ denote two arbitrary integers
	between $0$ and respectively $\ell_1$ and $\ell_2$.
	Since $b+(j+1)\le 2(b+j)$, Lemma \ref{lem:moments:sup} and
	Chebyshev's inequality together imply that
	\[
		\P\left\{ \sup_{x\in(a+i,a+i+1]}\sup_{t\in(b+j,b+j+1]} \left[\e^{-\beta t}
		u(t\,,x) \right] > 1\right\}
		\le \inf_{k\ge 2}\left[N_1^k\e^{2N_2k^3(b+j) - \beta k(b+j)}\right].
	\]
	Set $k:=\sqrt{\beta/(6N_2)}$ in the preceding infimization problem in order
	to see that
	\begin{align*}
		\P\left\{ \sup_{x\in(a+i,a+i+1]}\sup_{t\in(b+j,b+j+1]} \left[\e^{-\beta t}
			u(t\,,x) \right] > 1\right\}
			&\le N_1^{\sqrt{\beta/(6N_2)}}\exp\left( -\frac{2\beta^{3/2}
			(b+j)}{3\sqrt{6N_2}}\right)\\
		&\le N_1^{\sqrt{\beta/(6N_2)}}\exp\left( -\frac{2\beta^{3/2}
			b}{3\sqrt{6N_2}}\right),
	\end{align*}
	valid as long as $k=\sqrt{\beta/(6N_2)}\ge 2$. The lemma follows from
	adding the preceding expression from $j=0$ and $i=0$ to $j=\ell_2$
	and $i=\ell_1$.
\end{proof}

Proposition \ref{pr:sup:tail} paves the way for the main result of this
section, which is presented next. The following theorem complements
the lower bound of Theorem \ref{th:LB:main} by yielding a corresponding
almost-sure upper bound for the macroscopic Hausdorff dimension of
$S_\vartheta(\mathscr{P}(\beta))$. Recall the
universal constant $N_2$ from Lemma \ref{lem:moments:sup}.

\begin{theorem}\label{th:Dim:UB}
	For all $\beta\ge 24N_2$ and all $\vartheta>0$,
	\[
		\Dimh\left( \left\{ (x\,,t)\in\R\times(1\,,\infty):\,
		u(\vartheta\log t, x)>t^{\vartheta\beta}\right\}\right)
		\le \max\left\{ 1\,,2-N_4\beta^{3/2}\vartheta\right\} \qquad\text{a.s.}
	\]
\end{theorem}

\begin{proof}%[Proof of Theorem \ref{th:Dim:UB}]
	Clearly, for all $b\ge1$,
	\[
		\ell_2 := \vartheta\log(b+1)-\vartheta\log(b)
		=\vartheta\log\left(1+\frac 1b\right) \le \vartheta\log 2.
	\]
	Therefore, Proposition \ref{pr:sup:tail} implies that
	for all $a,\vartheta>0$, $b\ge 1$, and $\beta\ge24 N_2$,
	\begin{align*}
		&\P\left\{ \exists t\in(b\,,b+1]:\
			\sup_{x\in(a,a+1]} u(\vartheta\log t\,,x) >t^{\vartheta\beta}\right\}\\
		&\hskip2in=\P\left\{ \exists s\in\left( \vartheta\log b\,,
			\vartheta\log(b+1)\right]:\ \sup_{x\in(a,a+1]} u(s\,,x) >\e^{\beta s}\right\}\\
		&\hskip2in\le N_3\vartheta\log(2)\cdot
			b^{-N_4\beta^{3/2}\vartheta}.
	\end{align*}

	Choose and fix an arbitrary constant $q>1$. The
	preceding shows that, uniformly for all
	$b\in(\e^{n/q}\,,\e^{n+1}]$ and integers $n\ge 1$,
	\begin{equation}\label{eq:sup:tail}\begin{split}
		&\P\left\{ \exists t\in(b\,,b+1]:\
			\sup_{x\in(a,a+1]} u(\vartheta\log t\,,x) >t^{\vartheta\beta}
			\right\}\\
		&\hskip2.3in\le N_3\vartheta\log(2)\cdot
			\exp\left(-\left[\frac{N_4\beta^{3/2}\vartheta}{q} \right]n\right).
	\end{split}\end{equation}
	This is the key probability estimate required for the proof.
	
	Let us consider the random set
	\[
		\mathcal{G}_{\vartheta,\beta}:=
		\left\{ (x\,,t)\in(0\,,\infty)\times(1\,,\infty):\
		u(\vartheta\log t\,,x) > t^{\vartheta\beta} \right\}.
	\]
	Next we study the structure of
	$\mathcal{G}_{\vartheta,\beta}\cap\cS_n$ for all sufficiently-large
	integers $n\ge1$. Let us recall that $q>1$ is a fixed but arbitrary real
	number, and then observe that for all integers $n\ge 1$,
	
	\begin{equation}\label{GlL}
		\mathcal{G}_{\vartheta,\beta}\cap\cS_n \subseteq
		\ell_n \cup\mathcal{L}_n,
	\end{equation}
	where
	\begin{align*}
		\ell_n &:= \left( \e^n,\e^{n+1}\right]\times
			\left(0\,,\e^{n/q}\right]
			\cup \left(0\,,\e^{n/q}\right]\times\left( \e^n,\e^{n+1}\right] ,\\
		\mathcal{L}_n &:= \mathcal{L}_n(\vartheta\,,\beta) :=
			\mathcal{G}_{\vartheta,\beta}\cap
			\left( \e^{n/q}\,,\e^{n+1}\right]^2.
	\end{align*}
	Eq.\ \eqref{GlL} essentially decomposes $\mathcal{G}_{\vartheta,\beta}\cap\cS_n$
	into a ``little'' part $\ell_n$ and a ``large'' part $\mathcal{L}_n$.
	Because $q>1$ and
	\[
		\bigcup_{n=1}^\infty\ell_n \subseteq\left\{ (x\,,y)\in(0\,,\infty)^2:\
		y<x^{1/q}\right\}\cup\left\{
		(x\,,y)\in(0\,,\infty)^2:\
		y>x^{1/q}\right\},
	\]
	Corollary \ref{co:Epi} ensures that
	\begin{equation}\label{prev:ineq}
		\Dimh\left(\bigcup_{n=1}^\infty\ell_n\right) \le1.
	\end{equation}
	Owing to \eqref{GlL}, the inequality \eqref{prev:ineq} is good enough to imply that
	\begin{equation}\label{Dim:GlL}
		\Dimh\left(\mathcal{G}_{\vartheta,\beta}\right)
		\le \max\left\{ 1\,,\Dimh\left(\bigcup_{n=1}^\infty
		\mathcal{L}_n\right)\right\}\qquad\text{a.s.}
	\end{equation}
	
	We estimate an upper bound for the Hausdorff dimension of
	$\cup_{n=1}^\infty \mathcal{L}_n$
	as follows: There are at most $O(\e^{2n})$ squares of
	of the form $(a\,,a+1]\times(b\,,b+1]\subset
	(\e^{n/q}\,,\e^{n+1}]^2$, uniformly for all
	integers $n\ge 1$. We can cover each $\mathcal{L}_n$ with only such
	squares of the form $(a\,,a+1]\times(b\,,b+1]$ that additionally satisfy
	\begin{equation}\label{star}
		\sup_{x\in(a,a+1]} u(\vartheta\log t\,,x) >t^{\vartheta\beta}
		\quad\text{for some $t\in(b\,,b+1]$.}
	\end{equation}
	These remarks and \eqref{eq:sup:tail} together show that, for every integer
	$n\ge1$ and for all real numbers $\rho>0$,
	\begin{align*}
		\E\left[\nu_\rho^n\left(\mathcal{L}_n\right)\right]
			&\le \E\left(\sum_{\substack{(a,a+1]\times(b,b+1]\subseteq
			(\e^{n/q},\,\e^{n+1}]^2:\\
			\text{\eqref{star} holds}}} \left(\frac{1}{\e^n}\right)^\rho\right)\\
		&\le \text{const}\cdot\exp\left(
			-\left[ \frac{N_4\beta^{3/2}\vartheta }{q}-2 +\rho\right]n\right),
	\end{align*}
	where the implied constant depends only on $(\vartheta\,,\beta\,,N_2)$.
	In particular,
	\[
		\sum_{n=1}^\infty\nu^n_\rho(\mathcal{L}_n)<\infty\quad\text{a.s.,
		for every $\rho\in\left(2-\frac{N_4\beta^{3/2}\vartheta}{q}\,, 2 \right],$}\quad
	\]
	provided that $2 > q^{-1}N_4\beta^{3/2}\vartheta$. This, and the definition of
	Hausdorff dimension, together imply that
	\[
		\Dimh\left(\bigcup_{n=1}^\infty\mathcal{L}_n\right) \le
		\left(2-\frac{N_4\beta^{3/2}\vartheta}{q}\right)_+\qquad\text{a.s.},
	\]
	where $a_+:=\max(a\,,0)$ for all real $a$, as is usual.
	Because $q>1$ is arbitrary, and since
	the definition of $\mathcal{G}_{\vartheta,\beta}$
	does not depend on $q$, eq.\ \eqref{Dim:GlL} implies that
	\[
		\Dimh\left( \mathcal{G}_{\vartheta,\beta}\right)\le
		\max\left\{ 1\,, 2-N_4\beta^{3/2}\vartheta\right\}
		\qquad\text{a.s.}
	\]
	
	A symmetric argument implies that almost surely,
	\[
		\Dimh\left(\left\{ (x\,,t)\in(-\infty\,,0)\times(1\,,\infty):\
		u(\vartheta\log t\,,x) > t^{\vartheta\beta} \right\}\right)\le
		\max\left\{1\,,2-N_4\beta^{3/2}\vartheta\right\}.
	\]
	Therefore, the preceding two displayed inqualities together imply the
	theorem since
	\[
		\Dimh(\{0\}\times(1\,,\infty))=1,
	\]
	as can be checked
	from first principles or from an
	example of Barlow and Taylor \cite[\S4.1]{BT1989}.
\end{proof}

\subsection{Proof of Theorem \ref{th:SHE} and Corollary \ref{co:SHE}}
We complete this section by first deriving Theorem \ref{th:SHE} and then its
Corollary \ref{co:SHE}, in this order.

\begin{proof}[Proof of Theorem \ref{th:SHE}]
	Because the macroscopic dimension of a set is unaffected by local changes in that set,
	one can see readily that
	\[
		\Dimh\left\{ (x\,,t)\in\R\times(1\,,\infty):\
		u(\vartheta \log t\,,x)\ge t^{\vartheta\beta}\right\}
		= \Dimh\left[S_\vartheta(\mathscr{P}(\beta)) \right].
	\]
	Therefore, we can deduce Theorem \ref{th:SHE} readily from
	Theorems \ref{th:LB:main} and \ref{th:Dim:UB}.
\end{proof}

\begin{proof}[Proof of Corollary \ref{co:SHE}]
	It accord with Theorem \ref{th:SHE}, for every  $\beta_1 >b$ we can choose
	$\vartheta_1 \in (0\,, \varepsilon \beta_1^{-3/2})$ such that
	$1 < \Dimh\left[S_{\vartheta_1}(\mathscr{P}(\beta_1)) \right] <2$.
	Next we first choose $\beta_2 > \beta_1$, sufficiently close to
	$\beta_1$  that $\vartheta_1 < \varepsilon \beta_2^{-3/2}$,
	and then choose $\vartheta_2\in ( \vartheta_1\,, \varepsilon \beta_2^{-3/2})$.
	Theorem \ref{th:SHE} implies that
	$1 < \Dimh\left[S_{\vartheta_i}(\mathscr{P}(\beta_j)) \right] <2$ for $i, j \in \{1, 2\}.$
	Repetitive application of this procedure yields the
	first conclusion of Corollary  \ref{co:SHE}.
	
	To achieve the property in the second conclusion, we
	choose the sequences $\{\beta_i\}_{i=1}^\infty$ and $\{\vartheta_i\}_{i=1}^\infty$
	a little more carefully. For example,
	if we choose $\vartheta_1$ small, then we can choose $\beta_2$ and $\vartheta_2$ such that,
	in addition to the two properties mentioned above, they also satisfy
	\[
		\frac{A} a < \left(\frac{\beta_2}{\beta_1}\right)^{3/2} \frac{\vartheta_2}{\vartheta_1},
	\]
	where $A$ and $a$ are the constants in Theorem \ref{th:SHE}. In this way, we find that
	\[
		\Dimh\left[S_{\vartheta_2}(\mathscr{P}(\beta_2)) \right] <
		\Dimh\left[S_{\vartheta_1}(\mathscr{P}(\beta_1)) \right].
	\]
	An inductive application of this procedure yields two
	sequences $\{\beta_i\}_{i=1}^\infty$ and $\{\vartheta_i\}_{i=1}^\infty$
	such that $n\mapsto \Dimh[S_{\vartheta_n}(\mathscr{P}(\beta_n))]$
	is strictly decreasing. This completes the proof of the corollary.
\end{proof}

\section{A non-intermittent case}\label{sec:non:interm}

Let us now consider our stochastic PDE
\eqref{SHE} under the condition that the function $\sigma$ is constant.
In particular, $\sigma$ fails to satisfy the intermittency condition \eqref{cond:interm}.
For simplicity, we consider only the case that the initial function is identically zero
and $\sigma\equiv1$.
That is, we are interested in the random field $Z$   that solves
\begin{equation}\label{Z}
	\dot{Z}(t\,,x)=\frac{1}{2}Z''(t\,,x)+\xi(t\,,x)
	\qquad\text{for $(t\,,x)\in(0\,,\infty)\times \R$ with $Z(0)\equiv 0$}.
\end{equation}
It is well-known that $Z(t\,,x)$ has the following ``mild formulation'':
\[
	Z(t\,,x)= \int_{(0,t)\times\R} p_{t-s}(y-x) \xi (\d s\, \d y).
\]
See Walsh \cite[Chapter 3]{Walsh}.
In particular, the process $Z$ is a centered Gaussian process
with variance function
\begin{equation}\label{Var(Z)}
	\Var[Z(t\,,x)] = \int_0^t\d s\int_{-\infty}^\infty \d y\ |p_{t-s}(w-x)|^2
	=(t/\pi)^{1/2},
\end{equation}
for all $t>0$ and $x\in\R$.  Because the moments of $Z(t\,,x)$
are completely described by its variance, it follows that
$Z$ is not intermittent in the sense
that its moments do not grow exponentially
with time. As a result, one does not expect the exceptionally-large
peaks of $Z$ to be exponentially large in the time variable.
Still,  we proved a few years ago
\cite{KKX} that the spatial peaks of
$Z(t)$ form a multifractal for every fixed $t>0$.

The main result of this section is the following description of
the complex, multifractal nature of the tall spatio-temporal peaks of $Z$.
The following is the analogue of Theorem \ref{th:SHE} for the constant-coefficient,
linear SPDE \eqref{Z}.

\begin{theorem}\label{th:lin}
	For every $\beta,\varepsilon>0$,
	\[
		\Dimh\left\{ \left(x\,,\exp\left(t^{2\varepsilon}\right)\right)
		\in\R\times(\e \,,\infty):\
		Z(t\,,x)  > \frac{\beta t^{(1/4)+\varepsilon}}{\pi^{1/4}}  \right\}
		= \max\left\{1\,,2-  \frac{\beta^{2}}{2}\right\}
		\quad\text{a.s.}
	\]
\end{theorem}

\begin{remark}
	There is no canonical choice of how one can measure the heights of
	the very tall spatio-temporal peaks of $Z$. However,
	in light of \eqref{Var(Z)},
	the gauge function $t\mapsto t^{(1/4)+\varepsilon}$ is a natural choice.
	And Theorem \ref{th:lin} basically says
	that a certain ``stretching'' 	$[(x\,,t)\mapsto (x\,,\exp(t^{2\varepsilon}))]$
	of the random set
	\[
		\left\{(x\,,t)\in\R\times(\e\,,\infty):\
		Z(t\,,x) > \frac{\beta t^{(1/4)+\varepsilon}}{\pi^{1/4}} \right\}
	\]
	has dimension $\max\{1\,,2-\beta^2/2\}$ a.s.\ for every $\beta, \varepsilon>0$.
	By contrast with Theorem \ref{th:SHE}, however,  the stretch factor here
	is not arbitrary and depends on $\varepsilon$. In principle,
	however, it should be possible
	to produce similar results for various stretch factors.
\end{remark}

We conclude this section by presenting two prefatory results
about the solution $Z$ to \eqref{Z}. These results will be used
in the next 2 subsections in order to complete the proof of
Theorem \ref{th:lin}.

First, we observe the following immediate
consequence of the Gaussian nature of the law of $Z$;
see also \eqref{Var(Z)}.

\begin{lemma}\label{lem:lin:tail}
	There exists a constant $c>1$ such that
	\[
		\frac{1}{c\lambda}\, \e^{ -\lambda^2/2} \leq
		\P \left\{ Z(t\,,x) \geq (t/\pi)^{1/4} \lambda \right\}
		\leq \frac{c}{\lambda}\,  \e^{ -\lambda^2/2},
	\]
	uniformly for all $t>0$, $x\in\R$ and $\lambda>1$.
\end{lemma}

Next, we observe, using the semigroup property of the heat kernel, that
\[
	\Corr\left(Z(t\,,x), Z(t\,,y)\right) =  \sqrt{\frac\pi t}
	\int_0^t p_{2s} (x-y) \d s
	=\sqrt\pi\int_0^1 p_{2s}\left(\frac{x-y}{\sqrt t}\right)\d s.
\]
The stationary Gaussian process $x\mapsto Z(t\,,x)$ has the following properties:
\begin{compactenum}
\item[\bf P1.] For all $x\in\R$ and $\alpha\in(0\,,\nicefrac12)$,
	\[
		\lim_{t\to\infty}\Corr\left[ Z(t\,,x) \,, Z(t\,,x+ t^\alpha)\right] =
		\sqrt\pi\int_0^1 p_{2s}(0)\,\d s=1; \text{ and by contrast,}
	\]
\item[\bf P2.] For all $x\in\R$ and $\alpha>\nicefrac12$,
	\[
		0\le \Corr\left[Z(t\,,x) \,, Z ( t\,, x+ t^\alpha)\right]
		\le \exp\left( - \tfrac14 t^{2\alpha-1}\right),
	\]
	which rapidly tends to zero  as $t\to\infty$.
\end{compactenum}
Thus, we may combine {\bf P1} and {\bf P2} in order to
deduce the well-known informal assertion that
the ``correlation length'' of  the stochastic process
$x \mapsto Z(t\,,x)$ is $\sqrt{t}$. From this, and the rapid rate of convergence
to zero in {\bf P2}, one might surmise that
$Z(t\,,x)$ and $Z(t\,,y)$ are in fact \emph{asymptotically independent}
when $|x-y|\gg \sqrt{t}$.  Lemma \ref{lem:lin:indep} below
verifies this by giving  a rigorous meaning to
``asymptotic independence.''

For all $x\in \R$ and $B>0$, define
\[
	Z^{(B)}(t\,,x) := \int_{(0,t)\times[x-(Bt)^{1/2},\ x+(Bt)^{1/2}]}
	p_{t-s}(y-x)\,\xi(\d s\,\d y).
\]
The following lemma is essentially borrowed from Khoshnevisan,
Kim, and Xiao \cite{KKX}; see
Eq.\ (6.20) and Observation 1 of that paper ({\it loc.\ cit.}).

\begin{lemma}\label{lem:lin:indep}
	For all $t,B,\lambda>0$,
	\[
		\sup_{x\in\R}
		\P\left\{ \left| Z(t\,,x) - Z^{(B)}(t\,,x)\right|>\lambda\right\}
		\le 2\exp\left( -\frac{\lambda^2}{2}\sqrt{\frac{\pi}{8t}}\,
		\e^{B/2}\right).
	\]
	In addition, if $x_1,x_2,\ldots,x_m\in\R$ satisfy $|x_i-x_j|>2(Bt)^{1/2}$
	when $1\le i\ne j\le m$, then the random variables
	$Z^{(B)}_t(x_1),\ldots,Z^{(B)}_t(x_m)$ 	are independent.
\end{lemma}

Armed with these preliminary facts, we proceed with developing bounds for
the macroscopic Hausdorff dimension of the tall spatio-temporal peaks of the random
field $Z$.

\subsection{A lower bound}
The main result of this section is the following lower bound
on the macroscopic Hausdorff dimension of the
space-time set of extremely-tall peaks of height $f(t)= t^{(1/4)+\varepsilon}$
and level $\beta$. The following is
the precise statement that we will prove shortly.

\begin{proposition}\label{pr:LB:lin}
	For every $\beta,\varepsilon>0$,
	\[
		\Dimh\left\{ \left(x\,,\exp\left(t^{2\varepsilon}\right)\right)
		\in\R\times(\e \,,\infty):\
		Z(t\,,x)  > \frac{\beta t^{(1/4)+\varepsilon}}{\pi^{1/4}}  \right\}
		\ge \max\left\{1\,,2-  \frac{\beta^{2}}{2}\right\}
		\quad\text{a.s.}
	\]
\end{proposition}
Clearly, Proposition \ref{pr:LB:lin} implies half of the content
of Theorem \ref{th:lin}.

\begin{proof}%[Proof of Theorem \ref{pr:LB:lin}]
	We use the same procedure as in the proof
	of Theorem \ref{th:LB:main}. Consider the random set
	\begin{align*}
		\mathcal{G} :&= \left\{ \left( x\,,
			\exp\left(t^{2\varepsilon}\right)\right)   \in\R\times(\e \,,\infty):
			Z(t\,,x) > \frac{\beta t^{(1/4)+\varepsilon}}{\pi^{1/4}} \right\}\\
		& = \left\{ (x\,,t)\in\R\times(1 \,,\infty):
			Z\left( [\log t]^{1/(2\varepsilon)} \,,x\right) > \frac{\beta}{\pi^{1/4}}
			[\log t]^{(1/2)+[1/(8\varepsilon)]} \right\},
	\end{align*}
	whose macroscopic dimension is of interest to us.
	Define
	\[
		\mathcal{G}_0 := \mathcal{G} \cap \bigcup_{n=0}^\infty
		\left( \e^n,\e^{n+1}\right]^2.
	\]
	Since $\mathcal{G}_0\subseteq\mathcal{G}$ a.s., it suffices to prove
	that
	\[
		\Dimh(\mathcal{G}_0) \ge 2 - \frac{\beta^{2}}{2} \qquad
		\text{a.s.}
	\]
	
	Define, for all reals $\gamma\in(0\,,1)$ and integers $n\ge 0$,
	\[
		a_{j,n} := \e^n + j\e^{n\gamma}
		\qquad 0\le j<\e^{n(1-\gamma)},
	\]
	and
	\[
		\mathcal{I}_{j,n}(\gamma) :=
		\left( a_{j,n}\,, a_{j+1,n}\right]
		\qquad 0\le j<\e^{n(1-\gamma)}.
	\]
	For all such reals $\gamma$ and integers $n$
	and $j$, and for every $\varepsilon\in(0\,,1)$,
	we can find points $x_1,\ldots,x_m\in \mathcal{I}_{j,n}(\gamma)$---depending
	only on $(n\,,\gamma\,,\varepsilon\,,j)$---such that whenever $n>1$:
	\begin{itemize}
		\item[-] $|x_k-x_l|\ge \e^{n\varepsilon}$
			whenever $1\le k< l\le m$; and
		\item[-] $\frac12\e^{n(\gamma-\varepsilon)} \le m \le
			2\e^{n(\gamma-\varepsilon)}$.
	\end{itemize}
	We can write
	\begin{equation}\label{P<T_1T_2}
		\P\left\{\sup_{x\in \mathcal{I}_{j,n}(\gamma)} Z(t\,,x)
		\leq \frac{\beta t^{(1/4)+\varepsilon}}{\pi^{1/4}} \right\}\le
		T_1+T_2,
	\end{equation}
	where
	\begin{align*}
		T_1 &:=
			\P\left\{\sup_{x\in \mathcal{I}_{j,n}(\gamma)} Z^{(n)}(t\,,x)
			\leq \left(\frac{t}{\pi}\right)^{1/4}\beta
			\left[t^\varepsilon+1\right] \right\},\\
		T_2 &:= \P \left\{\sup_{x\in \mathcal{I}_{j,n}(\gamma)}
			\left| Z(t\,,x)-Z^{(n)}(t\,,x) \right| \geq \beta
			\left(\frac{t}{\pi}\right)^{1/4}\right\} .
	\end{align*}
	Thanks to Lemma \ref{lem:lin:tail} and Lemma \ref{lem:lin:indep},
	whenever the condition
	\begin{equation}\label{xx:cond}
		|x_k-x_l|\geq 2\sqrt{nt}
	\end{equation}
	holds, we can deduce that
	\begin{equation}\label{T1:UB}\begin{split}
		T_1 &\le  \left(\P\left\{ Z^{(n)}(t\,,0) \leq (t/\pi)^{1/4}\beta
			\left[ t^\varepsilon+1\right] \right\}\right)^m\\
		&\le \left(1- \P\left\{ Z^{(n)}(t\,,0) \geq
			\left(\frac{t}{\pi}\right)^{1/4}\beta
			\left[ t^\varepsilon+1\right]
			\right\}\right)^m\\
		&\le \left(1- \P\left\{ Z(t\,,0) \ge
			\left(\frac{t}{\pi}\right)^{1/4}\beta \left[ t^\varepsilon+2\right] \right\}
			+  2\exp\left[ -\frac{\beta^2\e^{n/2} }{4\sqrt{2}}\right]\right)^m\\
		&\le\left(1- \frac{\exp\left(-\beta^2\left( t^\varepsilon
			+2\right)^2/2\right)}{c\beta\left( t^\varepsilon+2\right)}
			+  2\exp\left[ -\frac{\beta^2\e^{n/2} }{4\sqrt{2}}\right]\right)^m.
	\end{split}\end{equation}
	[There is nothing special about the fact that the $x$-variable in
	$Z^{(n)}(t\,,0)$ and $Z(t\,,0)$ is chosen as $x=0$ since
	$Z^{(n)}(t,x)$ and $Z(t,x)$ are stationary random fields
	for every fixed $t>0$.]
	An elementary argument now shows that
	\begin{equation}\label{T2:UB}
		T_2 \le m\P \left\{  \left| Z(t\,,x)-Z^{(n)}(t\,,x) \right| \ge
		\beta\left(\frac{t}{\pi}\right)^{1/4}\right\}
		\le 2m\exp\left( -\frac{\beta^2\e^{n/2}}{4\sqrt{2}} \right).
	\end{equation}
	Now, let us plug the bounds \eqref{T1:UB} and \eqref{T2:UB} in
	\eqref{P<T_1T_2}, then replace $t$ by $[\log t]^{1/(2\varepsilon)}$ throughout,
	and use the previously-mentioned inequality,
	$\frac12\e^{n(\gamma-\varepsilon)} \le m \le 2\e^{n(\gamma-\varepsilon)}$, in
	order to obtain the following:
	\begin{align*}
		&\sup_{t\in(\e^n ,\e^{n+1}]}\max_{\substack{j\in\Z\\
			0\le j<\e^{n(1-\gamma)}}}
			\P\left\{\sup_{x\in \mathcal{I}_{j,n}(\gamma)}  
			Z\left( [\log t]^{1/(2\varepsilon)} ,x\right) > \frac{\beta}{\pi^{1/4}}
			[\log t]^{(1/2)+[1/(8\varepsilon)]} \right\}\\
		&\hskip2in\leq \exp\left(-\e^{n(\gamma-\varepsilon-(\beta^2/2)-o(n))
			} \right) +\exp\left( -\frac{\beta^2\e^{n/2}}{4\sqrt{2}}
			+n(\gamma-\varepsilon) \right).		
	\end{align*}
	A change of variables ---from $t$ to $[\log t]^{1/(2\varepsilon)}$---%
	justifies the
	asymptotic independence that is required for the preceding to hold.
	More precisely put, we have
	\[
		\min_{1\le k\neq l\le m}
		|x_k-x_l | \geq  \e^{n\varepsilon}\gg
		\sup_{t\in(\e^n,\,\e^{n+1}]} 2\sqrt{n [\log t]^{1/(2\varepsilon)}},
	\]
	which verifies that \eqref{xx:cond} holds after we replace $t$ by
	$[\log t]^{1/(2\varepsilon)}$.
		
	In any event, we can deduce from the preceding
	that, as $n\to\infty$,
	\begin{align*}
		&\P\left\{ \mathcal{G}_0 \cap
			\left[ \{t\}\times  \mathcal{I}_{j,n}(\gamma)\right] =\varnothing
			\text{ for some $0\le j< \e^{n(1-\gamma)}$ and
			$t\in\Z\cap\left(\e^n,\e^{n+1}\right]$}\right\}\\
		&\le\P\left\{ \min_{\substack{t\in\Z\\t\in(\e^n,\,\e^{n+1}]}}
			\min_{\substack{j\in\Z\\
			0\le j<\e^{n(1-\gamma)}}}
			\sup_{x\in \mathcal{I}_{j,n}(\gamma)}
			\frac{Z\left( [\log t]^{1/(2\varepsilon)}\,,x\right)}{%
			[\log t]^{\frac14+\frac{1}{8\varepsilon}}}
			< \frac{\beta}{\pi^{1/4}}\right\}\\
		&\le \exp\left(n(2-\gamma)-\e^{n(\gamma-\varepsilon-\beta^2/2-o(n))} \right)
			+\exp\left(n(2-\gamma) -\frac{\beta^2\e^{n/2}}{4\sqrt{2}}
			+n(\gamma-\varepsilon) \right).
	\end{align*}
	Thus, the Borel--Cantelli lemma implies that, as long as
	\[
		\varepsilon,\gamma\in(0\,,1)
		\quad\text{satisfy}\quad
		\gamma-\varepsilon>\beta^2/2,
	\]
	the following holds with probability one:
	\[
		\mathcal{G}_0 \cap \left[ \{t\}\times  \mathcal{I}_{j,n}(\gamma)\right]
		\neq\varnothing
		\text{ for all $0\le j< \e^{n(1-\gamma)}$ and
		$t\in\Z\cap\left(\e^n,\e^{n+1}\right]$,}
	\]
	for all but a finite number of integers $n\ge 1$.
	From here, the remainder of the proof follows exactly the same
	pattern as the one for its counterpart in Theorem \ref{th:LB:main};
	see \textbf{Conclusion A} [following shortly after \eqref{a.s.:LB}]
	and its justification.  We omit the remaining
	details.
\end{proof}

\subsection{An upper bound}
The main result of this section is the following upper bound
on the macroscopic Hausdorff dimension of the
space-time set of tall peaks of height $t^\varepsilon$ and level $\beta$.

\begin{proposition}\label{pr:UB:lin}
	For every $\beta,\varepsilon>0$,
	\[
		\Dimh\left\{ \left(x\,, \exp\left(t^{2\varepsilon}
		\right)\right)   \in\R\times(\e \,,\infty):\
		Z(t\,,x) > \frac{\beta t^{(1/4)+\varepsilon}}{%
		\pi^{1/4}}\right\}
		\le \max\left\{1\,,2-  \frac{\beta^{2}}{2}\right\}
		\qquad\text{a.s.}
	\]
\end{proposition}

Theorem \ref{th:lin} is manifestly a consequence of
Propositions \ref{pr:LB:lin} and \ref{pr:UB:lin}.

We first consider the following estimate of the tail probability.
This estimate is essential to the proof of Proposition \ref{pr:UB:lin}.

\begin{proposition}\label{pr:tail:lin}
	Let $f: (1\,,\infty) \to (1\,,\infty)$ be a strictly increasing function
	that satisfies $\lim_{t\to\infty} f(t)=\infty$.
	Then, for all $\ell\ge1$ there exists a finite constant $C=C(\ell)>1$
	such that
	\[
		\P\left\{\exists t \in (b\,, b+\ell]:\
		\sup_{x\in (a,a+1]} \frac{Z(t\,,x)}{(t/\pi)^{1/4}} \ge \beta f(t)
		 \right\} \leq 2\exp\left\{ -\frac{(\beta f(b))^2}{2}+
		 C\beta f(b)\left(\frac{\ell}{b^{1/4}}+1\right)\right\},
	\]
	uniformly for every $a\in\R$ and all sufficiently large $b>1$,
\end{proposition}

\begin{proof}
	It is well-known  that there exists a finite constant $C>0$
	such that for all $t>0$ and $x,y\in\R$,
	\[
		 \E\left(|Z(t\,,x)-Z(s,\, y)|^2\right)
		 \leq C\left( |t-s|^{1/2}+|x-y|\right).
	\]
	See, for example, \S3.3 of Khoshnevisan \cite{CBMS}.
	Since $Z$ is a Gaussian process,
	a quantitative form of the Kolmogorov continuity theorem
	\cite[Theorem C.6, p.\ 114]{CBMS} implies that
	\[
		C_1:=\sup_{a\in\R,b>1}
		\E \left( \sup_{x, y\in (a,a+1]}\sup_{s, t\in (b, b+1]} |Z(t\,,x) -Z(s,\,y)| \right)
		<\infty.
	\]
	Because of this, and the fact that the random variable $Z(t\,,x)$
	has a centered normal distribution
	with variance $\sqrt{t/\pi}$, we obtain the following: Uniformly for all
	$a\in\R$, $b>1$, and $\ell\ge 1$,
	\begin{align}\notag
		\E\left(\sup_{x\in (a,a+1]} \sup_{t\in (b, b+\ell]}
			\frac{Z(t\,,x)}{(t/\pi)^{1/4}} \right)
		& \leq \E\left(\sup_{x\in (a,a+1]} \sup_{t,s\in (b, b+\ell]}
			\frac{|Z(t\,,x)-Z(s\,,x)|}{(t/\pi)^{1/4}} \right) +
			\E\left| \frac{Z(b\,,a)}{(b/\pi)^{1/4}} \right|\\\notag
		&\leq (\pi/b)^{1/4}C_1\ell + 1\\
		&\le C_2 \left(\frac{\ell}{b^{1/4}}+1\right),
		\label{eq:supE:lin}
	\end{align}
	where $C_2:=\max(C_1\,,1)$ is a finite constant that is independent of $(a\,,b\,,\ell)$.
	With \eqref{eq:supE:lin} under way, we can easily complete the proof.
	
	Define
	\[
		X_{a,b,\ell}:=\sup_{x\in (a,a+1]}\sup_{t\in (b, b+\ell]} \frac{Z(t\,,x)}{(t/\pi)^{1/4}}.
	\]
	Then, clearly,
	\begin{align*}
		\P\left\{\exists t \in (b, b+\ell]:\
			\sup_{x\in (a,a+1]} \frac{Z(t\,,x)}{(t/\pi)^{1/4}} \geq \beta f(t) \right\}
			&\le \P\left\{\sup_{x\in (a,a+1]}\sup_{t\in (b, b+\ell]}
			\frac{Z(t\,,x)}{(t/\pi)^{1/4}} \geq \beta f(b) \right\} \\
		&=\P\left\{X_{a,b,\ell}-\E X_{a,b,\ell} \geq \beta f(b)-\E X_{a,b,\ell}  \right\}\\
		&\le \P\left\{\left| X_{a,b,\ell}-\E X_{a,b,\ell}\right|
			\geq \beta f(b)-\E X_{a,b,\ell}  \right\}.
	\end{align*}
	A standard appeal to the Borell, Sudakov--T'sirelson inequality
	now yields
	\[
		\P\left\{\exists t \in (b\,, b+\ell]:\
		\sup_{x\in (a,a+1]} \frac{Z(t\,,x)}{(t/\pi)^{1/4}} \geq \beta f(t) \right\}
		\leq  2\exp\left\{ -\tfrac12(\beta f(b)-\E X_{a,b,\ell})^2 \right\}.
	\]
	[For a readable account see Adler \cite[Chapter II]{Adler}.]
	As $t\uparrow\infty$, $f(t)$ increases strictly to $\infty$.
	Therefore, for all sufficiently large $b$,
	\[
		 \beta f(b) > C_2\left(\frac{\ell}{b^{1/4}}+1\right) \geq \E X_{a,b,\ell}.
	 \]
	 Therefore, \eqref{eq:supE:lin} implies the result.
\end{proof}

Armed with Proposition \ref{pr:tail:lin}, we next proceed with the proof
of Theorem \ref{pr:UB:lin}.

\begin{proof}[Proof of Proposition \ref{pr:UB:lin}]
	We follow the general procedure, and use the same notation, as in the proof
	of Theorem \ref{th:Dim:UB}. Recall the random set
	\begin{align*}
		\mathcal{G} :&= \left\{ \left(x\,, \exp\left(t^{2\varepsilon}
			\right)\right)   \in\R\times(\e \,,\infty):\
			Z(t\,,x) > \frac{\beta t^{(1/4)+\varepsilon}}{\pi^{1/4}}
			\right\}\\
		& = \left\{ (x\,,t)\in\R\times(1 \,,\infty):\
			Z\left( [\log t]^{1/(2\varepsilon)}\,,x\right) >
			\frac{\beta}{\pi^{1/4}}
			[\log t]^{(1/2)+[1/(8\varepsilon)]}  \right\}.
	\end{align*}
	
	From Proposition \ref{pr:tail:lin}, we have that for all sufficiently large $b$,
	\begin{align*}
		&\P\left\{ \exists t\in
			\left( [\log b]^{1/(2\varepsilon)},
			[\log (b+1)]^{1/(2\varepsilon)}\right]:\!\!
			\sup_{x\in(a,a+1]} Z(t\,,x) >
			\frac{\beta}{\pi^{1/4}} t^{(1/4)+\varepsilon}\right\}\\
		&\hskip3.7in
			\leq 2\exp\left(-\frac{\beta^2\log b}{2} + 
			\widetilde{C}\beta \sqrt{\log(b+1)}  \right),
	\end{align*}
	where $\tilde C$ is a finite and positive constant that
	depends only on the constant
	$C$ of Proposition \ref{pr:tail:lin}.
	
	Choose and fix an arbitrary constant $q>1$. The
	preceding implies that, uniformly for all $a\in \R$, 
	$b\in(\e^{n/q}\,,\e^{n+1}]$ and sufficiently large integers $n\ge 1$,
	\begin{equation}\label{eq:sup:tail:lin}\begin{split}
		&\P\left\{ \left(x\,, \exp\left(t^{2\varepsilon}
			\right)\right)\in\mathcal{G} \,\, 
			\text{for some $(x\,,t)\in (a\,,a+1]\times(b\,,b+1]$}  \right\}\\
		&\hskip3.6in
			\le 2\exp\left(-\frac{\beta^2 n}{2q}+\widetilde{C}\beta\sqrt{n+2}\right).
	\end{split}\end{equation}
	This is the key probability estimate required for the proof.
	
	We now follow the same pattern as we did
	in the proof of Theorem
	\ref{th:Dim:UB}. Let us recall that $q>1$ is a fixed but arbitrary real
	number, and then observe that for all integers $n\ge 1$,
	\[
		\mathcal{G} \cap\cS_n \subseteq\ell_n \cup\mathcal{L}_n,
	\]
	where
	\begin{align*}
		\ell_n &:= \left( \e^n,\e^{n+1}\right]\times
			\left(0\,,\e^{n/q}\right]\cup \left(0\,,\e^{n/q}\right]\times
			\left( \e^n,\e^{n+1}\right],\\
		\mathcal{L}_n &:=
			\mathcal{G}\cap \left( \e^{n/q},\e^{n+1}\right]^2.
	\end{align*}
	As we observed earlier
	in the proof of Theorem \ref{th:Dim:UB}, it is suffices to prove that
	\begin{equation}\label{Dim:GlL:lin}
		\Dimh\left(\mathcal{G}\right)
		\le \max\left\{ 1\,,\Dimh\left(\bigcup_{n=1}^\infty
		\mathcal{L}_n\right)\right\}\qquad\text{a.s.}
	\end{equation}
	We estimate an upper bound for the Hausdorff dimension of
	$\cup_{n=1}^\infty \mathcal{L}_n$
	as follows: There are at most $O(\e^{2n})$ squares of
	of the form $(a\,,a+1]\times(b\,,b+1]\subset
	(\e^{n/q},\e^{n+1}]^2$, uniformly for all
	integers $n\ge 1$. We can cover each $\mathcal{L}_n$ with only such
	squares of the form $(a\,,a+1]\times(b\,,b+1]$ that additionally satisfy
	\begin{equation}\label{star:lin}
		\sup_{x\in(a,a+1]} Z\left(
		[\log t]^{1/(2\varepsilon)}\,,x\right) >
		\frac{\beta}{\pi^{1/4}} t^{(1/4)+\varepsilon}
		\quad \text{for some $t\in (b\,, b+1]$}.
	\end{equation}
	These remarks and \eqref{eq:sup:tail:lin} together show that
	\begin{align*}
		\E\left[\nu_\rho^n\left(\mathcal{L}_n\right)\right]
			&\le \E\left(\sum_{\substack{(a,a+1]\times(b,b+1]\subset
			(\e^{n/q},\,\e^{n+1}]^2:\\
			\text{\eqref{star:lin} holds}}} \left(\frac{1}{\e^n}\right)^\rho\right)\\
		&\le \text{const}\cdot\exp\left(
			-\left[ \frac{\beta^2}{2q}+\rho-2+O(1/\sqrt{n})\right]n\right).
	\end{align*}
	It follows immediately from this bound and
	the Borel--Cantelli lemma that
	\[
		\Dimh\left(\bigcup_{n=1}^\infty\mathcal{L}_n\right) \le
		\left(2-\frac{\beta^2}{2q}\right)_+\qquad\text{a.s.},
	\]
	where $a_+:=\max(a\,,0)$ for all real $a$, as  before.
	Because $q>1$ is arbitrary, and since
	the definition of $\mathcal{G}$
	does not depend on $q$, Eq.\ \eqref{Dim:GlL:lin} implies that
	\[
		\Dimh\left( \mathcal{G}\right)\le
		\max\left\{ 1\,, 2-\frac{\beta^2}{2}\right\}
		\qquad\text{a.s.}
	\]
	The remainder of the proof is exactly the same as the one for
	Theorem \ref{th:Dim:UB}; that is, we use a symmetric argument and a
	general example of Barlow and Taylor \cite[\S4.1]{BT1989}.
\end{proof}

\bigskip\noindent\textbf{Acknowledgements.} Two of us [K.K. and D.K.] would like to thank
	the Mathematical Sciences Research Institute [Berkeley, CA] for providing us with
	a wonderful research environment in October 2015 [D.K.] and the Fall Semester
	[K.K.] of 2015.

\spacing{.8}

\bigskip
\small
\noindent\textbf{Davar Khoshnevisan} [\texttt{davar@math.utah.edu}].
	Department of Mathematics, University of Utah, Salt Lake City ,UT 84112-0090\\[.2cm]
\noindent\textbf{Kunwoo Kim} [\texttt{kunwoo@postech.ac.kr}]\\
\noindent Department of  Mathematics, Pohang University of Science and Technology (POSTECH), Pohang, Gyeongbuk, Korea 37673 \\[.2cm]
\noindent\textbf{Yimin Xiao} [\texttt{xiao@stt.msu.edu}].
	Dept.\  Statistics \&\ Probability,
	Michigan State University, East Lansing, MI 48824

\newpage\appendix\section{A table of universal constants}
\begin{table}[h!]\centering\small
\begin{tabular}{|c||c|}\hline
{\bf Constant}&{\bf Source}\\\hline\hline
$A$&Lemma \ref{lem:moments:loc}\\\hline
$A_0$&Lemma \ref{lem:moments:loc}\\\hline
$b_0$&Proposition \ref{pr:tails}\\\hline
$b_1$&Theorem \ref{th:LB:main}\\\hline
$B$&Lemma \ref{lem:moments:diff:loc}\\\hline
$B_0$&Lemma \ref{lem:moments:diff:loc}\\\hline
$c$&Lemma \ref{lem:lin:tail}\\\hline
$C$&Proposition \ref{pr:tail:lin}\\\hline
$C_0$&Lemma \ref{lem:u-uc}\\\hline
$C_1$&Proof of Proposition \ref{pr:tail:lin}\\\hline
$C_2$&Proof of Proposition \ref{pr:tail:lin}\\\hline
$\widetilde{C}$&Proof of Proposition \ref{pr:UB:lin}\\\hline
$D$&Proof of Lemma \ref{lem:moments:sup}\\\hline
$K$&Proposition \ref{pr:tails}\\\hline
$K_0$&Proposition \ref{pr:tails}\\\hline
$K_1$&Proposition \ref{pr:smallball}\\\hline
$L$&Proposition \ref{pr:tails}\\\hline
$L_0$&Proposition \ref{pr:tails}\\\hline
$L_1$&Proof of Lemma \ref{lem:moments:sup}\\\hline
$L_2$&Proof of Lemma \ref{lem:moments:sup}\\\hline
$L_3$&Proposition \ref{pr:sup:tail}\\\hline
$L_4$&Proposition \ref{pr:sup:tail}\\\hline
$\lip$&Proof of Proposition \ref{lem:moments:loc}\\\hline
$m_0$&Theorem \ref{th:localization}\\\hline
$m_1$&Proposition \ref{pr:smallball}\\\hline
$M$&Lemma \ref{lem:moments}\\\hline
$M_0$&Lemma \ref{lem:moments}\\\hline
$M_1$&Eq.\ \eqref{M_1}\\\hline
$N$&Lemma \ref{lem:moments}\\\hline
$N_0$&Lemma \ref{lem:moments}\\\hline
$N_1$&Lemma \ref{lem:moments:sup}\\\hline
$N_2$&Lemma \ref{lem:moments:sup}\\\hline
$N_3$&Proposition \ref{pr:sup:tail}\\\hline
$N_4$&Proposition \ref{pr:sup:tail}\\\hline
\end{tabular}
\caption{A table of universal constants}
\label{table}\end{table}

\end{document}